\documentclass[10pt,a4paper]{amsart}%
\usepackage{amsfonts,amsmath,amssymb}
\usepackage{hyperref}
\usepackage{amssymb,amsthm,amsxtra}
\usepackage[usenames]{color}
\usepackage{amscd}
\usepackage{amsthm}
\usepackage{amsfonts}
\usepackage{amssymb}
\usepackage{amsmath}
\usepackage{graphicx}
\usepackage{enumerate}%
\usepackage[all]{xy}

\newcommand*\IsoTo{%
  \xrightarrow[]{\raisebox{-0.5 em}{\smash{\ensuremath{\sim}}}}%
}

\newtheorem*{theorem*}{Theorem}
\newtheorem*{remark*}{Remark}
\newtheorem*{example*}{Example}
\newtheorem{lemma}{Lemma}[subsection]
\newtheorem{proposition}[lemma]{Proposition}
\newtheorem{remark}[lemma]{Remark}

\newtheorem{theorem}[lemma]{Theorem}
\newtheorem{definition}[lemma]{Definition}
\newtheorem{notation}[lemma]{Notation}
\newtheorem{corollary}[lemma]{Corollary}

\newtheorem*{conjecture*}{Conjecture}

\newtheorem{thm}[lemma]{Theorem}
\newtheorem{prop}[lemma]{Proposition}
\newtheorem{lem}[lemma]{Lemma}
\newtheorem{defn}[lemma]{Definition}
\newtheorem{notn}[lemma]{Notation}
\newtheorem{cor}[lemma]{Corollary}

\newtheorem{rem}[lemma]{Remark}

\newtheorem{introtheorem}{Theorem}

\newtheorem{introthm}[introtheorem]{Theorem}

\baselineskip 18pt \textwidth 16cm  \theoremstyle{plain}

\newcommand{\Hom}{\operatorname{Hom}}
\newcommand{\oH}{\operatorname{H}}

\newcommand{\eps}{\varepsilon}

\newcommand{\re}{\operatorname{Re}}
\renewcommand{\Im}{\operatorname{Im}}
\newcommand{\Ker}{\operatorname{Ker}}

\newcommand{\R}{{\mathbb R}}
\newcommand{\bR}{{\mathbb R}}
\newcommand{\C}{{\mathbb C}}

\newcommand{\proofend}{\hfill$\Box$\smallskip}

\newcommand{\Span}{{\operatorname{Span}}}

\newcommand{\Nash}{\operatorname{Nash}}

\newcommand{\T}{{\mathcal T}}

\newcommand{\alp}{{\alpha}}

\newcommand{\lam}{{\lambda}}

\newcommand{\Fre}{{Fr\'{e}chet \,}}

\newcommand{\cA}{{\mathcal{A}}}
\newcommand{\cB}{{\mathcal{B}}}
\newcommand{\cC}{{\mathcal{C}}}
\newcommand{\cD}{{\mathcal{D}}}
\newcommand{\cE}{{\mathcal{E}}}
\newcommand{\cF}{{\mathcal{F}}}
\newcommand{\cG}{{\mathcal{G}}}

\newcommand{\et}{{\'{e}tale }}

\newcommand{\g}{{\mathfrak{g}}}

\newcommand{\fg}{{\mathfrak{g}}}
\newcommand{\fh}{{\mathfrak{h}}}

\newcommand{\fk}{{\mathfrak{k}}}
\newcommand{\fv}{{\mathfrak{v}}}
\newcommand{\h}{{\mathfrak{h}}}

\newcommand{\Supp}{\mathrm{Supp}}

\newcommand{\cO}{{\mathcal{O}}}
\newcommand{\GL}{\operatorname{GL}}

\newcommand{\Sym}{\operatorname{Sym}}

\newcommand{\Sc}{{\mathcal S}}

\newcommand{\ctp}{\widehat{\otimes}}

\newcommand{\cM}{\mathcal{M}}

\newcommand{\Chi}{\mathfrak{X}}

\newcommand{\fri}{\mathfrak{i}}
\newcommand{\fp}{\mathfrak{p}}
\newcommand{\fl}{\mathfrak{l}}

\newcommand{\fr}{\mathfrak{r}}

\newcommand{\fu}{\mathfrak{u}}

\newcommand{\cT}{\mathcal{T}}

\newcommand{\oE}{\widetilde{E}}
\newcommand{\oPhi}{\widetilde{\Phi}}

 \newcommand{\Rami}[1]{{#1}}
 \newcommand{\RamiA}[1]{{#1}}
 \newcommand{\DimaA}[1]{{#1}}
\newcommand{\p}{{\mathfrak{p}}}
\newcommand{\onto}{{\twoheadrightarrow}}

\newcommand{\ot}{\leftarrow}

\begin{document}

\author{Avraham Aizenbud}
%\address{Avraham Aizenbud, Faculty of Mathematics
%and Computer Science, The Weizmann Institute of Science POB 26,
%Rehovot 76100, ISRAEL.}
\email{aizenr@gmail.com}
\address{Avraham Aizenbud,
Faculty of Mathematics and Computer Science,
Weizmann Institute of Science,
POB 26, Rehovot 76100, Israel}
\urladdr{\url{http://www.wisdom.weizmann.ac.il/~aizenr}}

\author{Dmitry Gourevitch}
\email{dmitry.gourevitch@weizmann.ac.il}
\address{Dmitry Gourevitch,
Faculty of Mathematics and Computer Science,
Weizmann Institute of Science,
POB 26, Rehovot 76100, Israel}
\urladdr{\url{http://www.wisdom.weizmann.ac.il/~dimagur}}

\author{Siddhartha Sahi}
\email{sahi@math.rugers.edu}
\address{Siddhartha Sahi, Department of Mathematics, Rutgers University, Hill Center -
Busch Campus, 110 Frelinghuysen Road Piscataway, NJ 08854-8019, USA}

\date{\today}
\title{Twisted homology for the mirabolic nilradical}
% \title{Derivatives for smooth representations of $GL(n,{\mathbb{R}})$ and $GL
% (n,{\mathbb{C}})$ II - analytic aspects}


\begin{abstract}
The notion of derivatives for smooth representations of $GL(n,{\mathbb{Q}}_p)$ was
defined in \cite{BZ-Induced}.
In the archimedean case, an analog of the highest derivative was defined for irreducible
unitary representations in \cite{Sahi-Kirillov} and called the ``adduced" representation. In \cite{AGS1}
derivatives of all orders were defined for smooth admissible Fr\'{e}chet representations (of
moderate growth).

A key ingredient of this definition is the functor of twisted coinvariants with respect to the nilradical of the mirabolic subgroup.
In this paper we prove exactness of this functor and compute it on a certain class of representations. This implies  exactness of the highest derivative functor, and allows to compute highest derivatives of all monomial representations.
%and all unitary representations.}

In \cite{AGS1} these results are applied to finish the computation of adduced representations for all irreducible unitary representations and to prove uniqueness of degenerate Whittaker models for unitary representations.
%, thus completing the results of \cite{Sahi-Kirillov,Sahi-PAMS,SaSt,GS}}.
\end{abstract}

\maketitle
%\tableofcontents

%\address{Avraham Aizenbud, Faculty of Mathematics
%and Computer Science, The Weizmann Institute of Science POB 26,
%Rehovot 76100, ISRAEL.}

%\date{\today}

%
%2010 Mathematics Subject Classification:
% 20G20         Linear algebraic groups over the reals, the complexes, the quaternions
% 22E30         Analysis on real and complex Lie groups
% 22E45         Representations of Lie and linear algebraic groups over real fields: analytic methods

\section{Introduction}

\setcounter{lemma}{0}

The notion of derivative was first defined in \cite{BZ-Induced} for smooth
representations of $\ G_{n}=GL(n)$ over non-archimedean fields and became a
crucial tool in the study of this category.
The definition of derivative is based on the \textquotedblleft
mirabolic\textquotedblright\ subgroup $P_{n}$ of $G_{n}$ consisting of
matrices with last row $(0,\dots,0,1)$. The unipotent radical of this subgroup
is an $\left( n-1\right) $-dimensional linear space that we denote $V_{n}$,
and the reductive quotient is $G_{n-1}$. The group $G_{n-1}$ has 2 orbits on
$V_{n}$ and hence also on $V_{n}^{\ast }$: the zero and the non-zero orbit.
The stabilizer in $G_{n-1}$ of a non-trivial character $\psi $ of $V_n$ is isomorphic to $P_{n-1}$.

The construction of derivative  is based on two
functors: $\Phi$ and $\Psi$. This paper deals only with the  functor $\Phi$ of (normalized) twisted zero homology (in other words co-equivariants) with respect to $(V_{n},\psi )$, which goes from the category of smooth representations of $P_{n}$ to the
category of smooth representations of $P_{n-1}$.

To be more precise, fix  an archimedean local field $F$ and denote $G_n:=\GL(n,F)$. Let $\fp_n, \fv_n,..$ denote the Lie algebras of the corresponding Lie groups.  Let $\psi_{n}$ be the standard non-degenerate character of $V_{n}$, given by $\psi _{n}(x_1,\dots,x_{n-1}):=\exp(\sqrt{-1}\pi \re x_{n-1}).$ We will also denote by $\psi_n$ the corresponding character of the Lie algebra $\fv_n$.
For all $n$ and for all representations $\pi $ of ${\mathfrak{p}}_{n}$, we define
$$\Phi(\pi):=|\det|^{-1/2} \otimes \pi_{(\fv_n,\psi_n)}:=|\det|^{-1/2} \otimes \pi/\Span\{\alp v - \psi_n(\alp)v \, : \, v \in \pi, \, \alp \in \fv_{n}\}$$
%and
%$$\Psi(\pi):= \lim_{\overset{\longleftarrow }{l}} \pi /\Span\{\beta v\,|\,v  \in \pi ,\beta \in ({\mathfrak{v}}_{n})^{\otimes l} \}.$$
Now, define $E^k(\pi):=\Phi^{k-1}(\pi)$.
This is one of the three notions of derivatives defined in \cite{AGS1}. The other two are obtained from this one by applying coinvariants (respectively generalized coinvariants) with respect to $\fv_{n-k+1}$. However, in this paper we will only prove results concerning $E^k$.

In the non-archimedean case there is an equivalence between the category of smooth representations of $P_{n}$ and the category of $%
G_{n-1}$-equivariant sheaves on $V_{n}^{\ast }$. This equivalence is based
on the Fourier transform. Under this equivalence, $\Phi$ becomes the fiber at the point $\psi $. It can also be viewed as the composition of two functors:
restriction to the open orbit and an equivalence of categories between
equivariant sheaves on the orbit and representations of the stabilizer of a point. We used this as an intuition for the archimedean case. In particular, we expected $\Phi$ to be exact, which indeed is one of the main results of this paper, that we will now formulate.

 Let $\mathcal{M}(G_{n})$ denote the category of smooth
admissible {Fr\'{e}chet} representations of moderate growth.

\begin{introthm} \label{thm:ExactHaus}

For any $0 <k \leq n$
\begin{enumerate}
\item \label{it:Exact}
 $E^{k}:\mathcal{M}(G_{n}) \to \mathcal{M}(\fp_{n-k+1})$ is an exact functor.
\item \label{it:Haus} For any $\pi \in \cM(G_n)$, the natural (quotient) topology on $E^{k}(\pi)$ is Hausdorff. In other words,  $\fu_n^k(\pi \otimes (-\psi_n^k))$ is closed in $\pi$ (see \S \ref{subsec:PrelNot} for the definitions of $\fu_n^k$ and $\psi_n^k$) .
\end{enumerate}
\end{introthm}
For proof see \S \ref{sec:ExactHaus}.

\begin{introthm}\label{thm:Prod}
Let $n=n_1+\cdots+n_k$ and let $\chi_i$ be characters of $G_{n_i}$. Let $\pi= \chi_1 \times \cdots \times \chi_k$ denote the corresponding monomial representation. Then %$depth(\pi) = k$ and
$$E^k(\pi) = E^1(\chi_1) \times \cdots \times E^1(\chi_k) = ((\chi_1)|_{G_{n_1-1}} \times \cdots \times (\chi_k)|_{G_{n_k-1}}).$$
\end{introthm}

We prove this theorem in \S \ref{sec:PfProd}.

In the case of $k=n$ this becomes the Whittaker functor, which is known to be exact by \cite{CHM}. In this case Theorem \ref{thm:Prod} implies that the space of Whittaker functionals on a principal series representation is one-dimensional, which was proven for $G_n$ by Casselman-Zuckerman and for a general quasi-split reductive group by Kostant.

\subsection{Structure of our proof}

\subsubsection{Exactness and Hausdorffness}

Note that $\Phi(\pi)=|\det|^{-1/2} \otimes H_0(\fv_n,\pi \otimes \psi_n)$.
It turns out to be convenient to study it together with higher homologies $H_i(\fv_n,\pi \otimes \psi_n)$.
In particular, we prove the exactness of $\Phi$ and the Hausdorffness of $\Phi(\pi)$
together. Let us describe the proof of exactness, and the proof of
Hausdorffness will be incorporated into this argument.

The ideas of the proof come from two sources: the proof of the case $k=n$
given in \cite{CHM} and the analogy with the p-adic case, where one can use
the language of equivariant $l$-sheaves.

In \cite{CHM}, the first step of the proof is reduction to the statement
that the principal series representations are acyclic, using the Casselman
embedding theorem. This step works in our case as well. Then in \cite{CHM}
they prove that the principal series representations are acyclic by
decomposing them with respect to the decomposition of the flag variety into $%
U_n^n$ orbits (see \S \ref{subsec:PrelNot} for the definition of $U_n^k$ for $1 \leq k \leq n$). In their case, there is a finite number of $U_n^n$ orbits and
each orbit corresponds to an acyclic representation. In our case, $U_n^k$
has infinitely many orbits on the flag variety and some of them give rise to
non-acyclic representations.

Eventually we show that these problematic orbits do not contribute to the
higher $(\mathfrak{u}_{n}^{k},\psi _{n}^{k})$-homologies of principal
series since they are immersed in smooth families of orbits most of which
give rise to acyclic representations. It is difficult to see that directly,
therefore we turn to the strategy inspired by equivariant sheaves.

If we had
a nice category of $P_{n}$-representations, which like in the $p$-adic case
was equivalent to a nice category of $G_{n-1}$-equivariant sheaves on $%
V_{n}^{\ast }$, then the functor $\Phi  $ would be a composition of a
restriction to an open set and an equivalence of categories, and thus it
should be exact. Unfortunately it seems too hard at this point to define
such categories. Rather than doing this we define a certain class of $P_{n}$%
-representations. We define it using specific examples and constructions and
not by demanding certain properties. We do not view this class as a
\textquotedblleft nice category of representation of $P_{n}$", in particular
we do not analyze it using equivariant sheaves. However it is sufficient for
the proof of exactness. Namely we prove that this class includes the
principal series representations of $G_{n}$. We prove that it is closed
under $\Phi  $. Finally, we prove that its objects are acyclic with
respect to $\Phi  $.

In order to do this we develop some tools for computing the homology of
representations constructed in a geometric way.

\subsubsection{Derivative of monomial representations}

In the non-archimedean case, \cite{BZ-Induced} provides a description of any derivative of a Bernstein-Zelevinsky product $\rho_1\times \rho_2$ in terms of derivatives of $\rho_1$ and $\rho_2$. This description immediately implies the natural analog of Theorem \ref{thm:Prod}. This description is proved using geometric analysis of the  Bernstein-Zelevinsky product as the space of sections of the corresponding sheaf on the corresponding Grassmanian.

This method is hard to translate to the archimedean case for the following reasons.
\begin{itemize}
\item The analogous statement for the derivative as we defined it here does not hold. If we would not replace the functor of coinvariants by the functor of generalized coinvariants then the analogous statement might hold. However, we would lose the exactness, which played a crucial role in the proof in \cite{BZ-Induced}.

\item We do not have an appropriate language of infinite-dimensional sheaves.
\end{itemize}
For these reasons, we prove a product formula only for the depth derivatives, and only for products of characters. The drawback of that approach is that we have to consider product of more than two characters.

We prove the product formula using the geometric description of monomial representations. More specifically, we decompose the corresponding flag variety into $P_n$-orbits. This decomposition gives a filtration of the monomial representation. We compute the functor $\Phi$ on each of the associated graded pieces in geometric terms, and use the exactness of $\Phi$  to proceed by induction.

\subsection{Tools developed in this paper}

\subsubsection{The relative Shapiro lemma}

In the proof of the exactness of the derivatives we are interested in the study of homology of representations of the
following type. Let a real algebraic group $G$ act on a real algebraic
manifold $X$, and let $\cE$ be a $G$-equivariant bundle over $X$. Consider the
space ${\mathcal{S}}(X,\cE)$ of Schwartz sections of this bundle (see \S
\ref{sec:ScNash}). This is a representation of $G$. For technical reasons
we prefer to study Lie algebra homology rather than Lie group homology.
Thus we will impose some homotopic assumptions on $G$ and $X$ so that there
will be no difference. In the case when $X$ is homogeneous, i.e. $X=G/H$, a
version of the Shapiro Lemma states, under suitable homotopic and unimodularity assumptions on $G$ and $H$, that $\oH_*({\mathfrak{g}},{\mathcal{S}}(X,\cE)) =
\oH_*({\mathfrak{h}},\cE_{[1]})$,
where $\cE_{[1]}$ denotes the fiber of $\cE$ at
the class of $1$ in $G/H$ (cf. \cite{AGRhamShap}).

In Appendix \ref{app:PfShapLem} we prove a generalization of this statement
to the relative case. Namely, let $G$ act on $X \times Y$ and $\cE$ be a $G$%
-equivariant bundle over $X \times Y$. Suppose that $X=G/H$. Then, under the same assumptions on $G$ and $H$, we have
\begin{equation*}
\oH_*({\mathfrak{g}},{\mathcal{S}}(X \times Y,\cE)) = \oH_*({\mathfrak{h}},{%
\mathcal{S}}([1]\times Y,\cE_{[1] \times Y})).
\end{equation*}
For the zero homology this statement is in fact Frobenius descent (cf. \cite[Appendix B.3]{AGHC}).

\subsubsection{Homology of families of characters}

In the proof of the exactness of the derivatives, we use the relative Shapiro lemma to reduce to the study of
$\oH_*({\mathfrak{g}},{\mathcal{S}}(X,\cE))$, where the action of $G$ on $X$ is
trivial. Moreover, in our case $\cE$ is a line bundle and ${\mathfrak{g}}$ is
abelian. In \S \ref{sec:PfGeoGood} we prove that ${\mathcal{S}}(X,\cE)$ is acyclic
and compute $\oH_0({\mathfrak{g}},{\mathcal{S}}(X,\cE))$ under certain
conditions on the action of $G$ on $\cE.$

\Rami{\subsubsection{Tempered equivariant bundles}
Since we discuss non-algebraic (e.g. unitary) characters, it might become necessary to consider non-algebraic geometric objects. We tried to avoid that, but had to consider vector bundles of algebraic nature with a non-algebraic action of a group. A similar difficulty arose in \cite{AGST}, and the notion of tempered equivariant bundle was defined in order to deal with it. In this paper we develop more tools to work with such bundles. }

\subsection{Structure of the paper}$ $\\
%In \S \ref{sec:Prel} we give the necessary preliminaries on Harish-Chandra modules, admissible smooth representations, annihilator varieties and Bernstein-Zelevinsky product.

%\Rami{
The proof of exactness and Hausdorffness of $\Phi^k$ and the computation of the depth derivative of the monomial representations are based on the theory of Schwartz functions on Nash manifolds. In \S\ref{sec:ScNash} we review this theory, including the notion of equivariant tempered bundle.

In \S\ref{sec:ExactHaus} we prove that $\Phi^k$ is an exact functor on the category of smooth admissible representations and that $\Phi^k(\pi)$  is Hausdorff for any reducible smooth admissible representation $\pi$. In \S\ref{subsec:PfClass} we introduce the class of good representations of $\p_n$. The construction of this class is based on representations with simple geometric descriptions that we call geometric representations.  We prove that this class  includes the
principal series representations of $G_{n}$, is closed under $\Phi$ and  that its objects are acyclic with respect to $\Phi$.
In order to prove the last two statements we use Lemma \ref{lem:GeoGood} \RamiA{(the key lemma), which is proved} in \S\ref{sec:PfGeoGood}.

In \S\ref{sec:PfProd} we compute the depth derivative of a monomial representation. This computation is based on the  computation of  certain geometric representations. We formulate the results of this computation in \S\ref{subsec:ProdDer},  and prove them in \S\ref{subsec:PfLemProd}.

%In the section we wanted to add a remark that we also say something on other derivatives

In \S\ref{sec:PfSmoothLem} we give the postponed proofs of the technical lemmas that were used in sections  \ref{sec:ExactHaus} and \ref{sec:PfProd}.
For this we need some preliminaries on homological algebra, topological linear spaces and Schwartz functions on Nash manifolds, which we give in subsections \ref{subsec:PrelHomAlg}, \ref{subsec:PrelTopLin}, and \ref{subsec:PrelScNash}.

In \S\ref{sec:PfGeoGood} we prove the key lemma (Lemma \ref{lem:GeoGood}).

In Appendix \ref{app:PfShapLem} we prove  a version of the Shapiro lemma which is crucial for the proof of Lemma \ref{lem:GeoGood}.

\subsection{Acknowledgements}

We cordially thank Joseph Bernstein
%Semyon
%Alesker, Joseph Bernstein, Jim Cogdell, Laurent Clozel, Maria Gorelik, Anthony Joseph, Peter Trapa and David Vogan
for fruitful discussions, Gang Liu for his remarks and the referee for very careful proofreading and for his remarks.

Part of the work on this paper was done during the program ``Analysis on Lie Groups" at the Max Planck Institute for Mathematics (MPIM) in Bonn.
We would like to thank the organizers of the programm, Bernhard Kroetz,  Eitan Sayag and Henrik Schlichtkrull, and the administration of the MPIM for their hospitality.

A.A. was partially supported by NSF grant DMS-1100943 and ISF grant 687/13;

D.G. was partially supported by ISF grant 756/12 and a Minerva foundation grant.

\section{Preliminaries} \label{sec:Prel}

\subsection{Notation and conventions}\label{subsec:PrelNot}

\begin{itemize}
\item We will denote real algebraic groups by capital Latin letters and their complexified Lie algebras by small Gothic letters.

\item Let $\g$ be a complex Lie algebra. We denote by $\cM(\g)$ the category of (arbitrary) $\g$-modules. Let $\psi$ be a character of $\g$.
For a module $M \in \cM(\g)$ denote by $M^\g$ the space of $\g$-invariants, by $M^{\g,\psi}$ the space of $(\g, \psi)$-equivariants, by $M_{\g}$ the space of coinvariants, i.e. $M_{\g}:=M/\g M$ and by $M_{\g,\psi}$ the space of $(\g,\psi)$-coequivariants, i.e. $M_{\g,\psi}:=(M \otimes (-\psi))_{\g}$.
% \item We also denote by $M_{gen, \g}$ the space of the generalized co-invariants, i.e. %
% $$M_{gen, \g}=\lim_{\overset{\longleftarrow}{l}} M/\Span(\{\alp v\, |\, %v \in M, \alp \in (\g)^{\otimes l}\}).$$

%Throughout the section we will use the following notation.
%\begin{notn}
\item  $G_n :=GL(n,F)$, we embed $G_n \subset G_m$ for any $m>n$ by sending any $g$ into a block-diagonal matrix consisting of $g$ and $Id_{m-n}$. We denote the union of all $G_n$ by $G_{\infty}$ and all the groups we will consider will be embedded into $G_{\infty}$ in a standard way.
\item By a composition of $n$ we mean a tuple $(n_1 , \dots ,  n_k)$ of natural numbers such that $n_1+\cdots+n_k=n$. By a partition we mean a composition which is non-increasing, i.e. $n_1 \geq \cdots \geq n_k$.

\item We denote by $P_{n}\subset G_{n}$ the mirabolic subgroup (consisting of matrices with last row $(0 , \dots ,  0,1)$).

\item For a composition $\lambda=(n_1 , \dots ,  n_k)$ of $n$ we denote by $P_{\lam}$ the corresponding block-upper triangular parabolic subgroup of $G_n$.
For example, $P_{(1,\dots,1)}=B_n$ denotes the standard Borel subgroup, $P_{(n)}=G_n$ and $P_{(n-1,1)}$ denotes the standard maximal parabolic subgroup that includes $P_n$.

\item Let $V_n \subset P_n$ be the unipotent radical. Note that $V_n \cong F^{n-1}$ and $P_n = G_{n-1} \ltimes V_n$. Let $U_n^k := V_{n-k+1} V_{n-k+2} \cdots V_{n}$ and $S_n^k := G_{n-k} U_n^k$. Note that $U_n^k$ is the unipotent radical of $S_n^k$. Let $N_n:=U_n^n$.
\item Fix a non-trivial unitary additive character $\theta$ of $F$, given by $\theta(x)=\exp(\sqrt{-1}\pi \re x)$.
\item Let $\bar{\psi}_n^k:U_n^k \to F$ be the standard non-degenerate homomorphism, given by $\bar{\psi}_n^k(u)=\sum_{j=n-k}^{n-1} u_{j,j+1}$ and let $\psi_n^k:=\theta \circ \bar{\psi}_n^k$.%character of $U_n^k$.
\end{itemize}

We will usually omit the $n$ from the notations $U_n^k$ and $S_n^k$, and both indexes from $\psi_n^k$.

\begin{defn}\label{def:main}
Define the functor $\Phi: \cM(\fp_n) \to \cM(\fp_{n-1})$ by $\Phi(\pi):=\pi_{\fv_n,\psi} \otimes |\det|^{-1/2}$.% and $\Psi,\Psi_0: \cM(\fp_n) \to \cM(\fg_{n-1})$ by $\Psi(\pi):= \pi_{gen,\fv_{n}}$ and $\Psi_0(\pi):= \pi_{\fv_{n}}$.

For a $\p_n$-module $\pi$ we define derivative by
 $E^k(\pi):=\Phi^{k-1}(\pi):=\pi_{\fu^{k-1},\psi^{k-1}}\otimes |\det|^{-(k-1)/2}.$ Clearly it has a structure of a $\p_{n-k+1}$ - representation.
%Here $\p_{n-k+1}$ is the Lie algebra of the ${n-k+1}$-mirabolic.
For convenience we will also use untwisted versions of the above functors, defined by $\oPhi(\pi):=\Phi(\pi) \otimes |\det|^{1/2}$, and $\oE^k(\pi):=E^k(\pi) \otimes |\det|^{(k-1)/2}$.
\end{defn}

\subsection{Smooth representations}\label{subsec:HC}

%In this subsection we fix a real reductive group $G$, a minimal parabolic subgroup of $P \subset G$, and let $N$ denote the nilradical of $P$.
%We also fix a maximal compact subgroup $K \subset G$. Let $\fg,\fn,\fk$ denote the complexified Lie algebras of $G,N,K$, and let $Z_{G}:=U({\mathfrak{g}})^{G}$.
%
%
%
%\begin{thm}[Casselman-Wallach, see \cite{Wal2}, \S\S\S 11.6.8] \label{thm:CW}
%The functor $HC:\cM_{\infty}(G) \to \cM_{HC}(G)$ is an equivalence of categories.
%\end{thm}

%In fact, Casselman and Wallach construct an inverse functor $\Gamma: \cM_{HC}(G) \to \cM_{\infty}(G)$, that is called Casselman-Wallach globalization functor (see \cite[Chapter 11]{Wal2} or \cite{CasGlob} or, for a different approach, \cite{BerKr}).
% and $\Gamma: \cM_{HC}(G) \to \cM_{\infty}(G)$ the (inverse) Casselman-Wallach globalization functor.

\begin{thm}[Casselman-Wallach] \label{cor:CW}
$ $
\begin{enumerate}[(i)]
\item  \label{it:Ab} The category $\cM(G)$ is abelian.
\item \label{it:ClosIm} Any morphism in $\cM_{}(G)$ has closed image.
\end{enumerate}
\end{thm}

\begin{thm}[Casselman subrepresentation theorem, see \cite{CM}, Proposition 8.23]\label{thm:CasSubRep}
Let $\pi$ be a finitely generated admissible $(\g,K)$-module and $P$ be a minimal parabolic subgroup of $G$. Then there exists a
\RamiA{finite-dimensional}
%n irreducible finite-dimensional %smooth
representation $\sigma$ of $P$ such that $\pi$ may be imbedded into $Ind_P^G(\sigma)$.
\end{thm}

\subsection{Parabolic induction and Bernstein-Zelevinsky product} \label{subsec:ParInd}

Let $G$ be a real reductive group, $P$ be a parabolic subgroup, $M$ be the Levi quotient of $P$ and $pr:P \to M$ denote the natural map.

\begin{notn}
$ $
\begin{itemize}
\item For a Lie group $H$ we denote by $\Delta_H$ the modulus character of $H$, i.e. the absolute value of the determinant of the infinitesimal adjoint action.
\item For $\pi \in \cM(M)$ we denote by $I_P^G(\pi)$ the normalized parabolic induction of $\pi$, i.e. the space of smooth functions $f :G \to \pi$ such that $f(pg) = \Delta_P(p)^{1/2} \pi(pr(p))f(g)$, with the action of $G$ given by $(I_P^G(\pi)(g)f)(x):=f(xg)$.
\end{itemize}
\end{notn}
%\Dima{we induce on the LEFT!!}

%The behavior of annihilator variety under parabolic induction is described by the following theorem.
%
%\begin{thm}\label{thm:AnnVarProd}
%Note that we have a natural embedding $\fm^* \hookrightarrow \fp^*$ and a natural projection $r:\g^* \to \fp^*$. Let $\pi \in \cM_{\infty}(M)$. Then $\cV(I_M^G(\pi)) = G_{\C}\cdot r^{-1}(\cV(\pi))$, where $G_{\C}$ is the complexification of $G$.
%\end{thm}
%This theorem is well-known and can be deduced from \cite[Theorem 2]{BB2}.
%\Dima{ There is a problem. In \cite{BB2} $G$ is probably the adjoint group of $\g$. Is it the same as our $G_\C$?}

\subsubsection{Bernstein-Zelevinsky product}

We now introduce the Bernstein-Zelevinsky product notation for parabolic induction.

\begin{definition}\label{def:BZProd}
If $\alpha=(n_{1} , \dots ,    n_{k})$ is a composition of $n$ and
$\pi_{i}\in\cM(G_{n_{i}})$ then $\pi_{1}\otimes\cdots\otimes\pi_{k}$ is
a representation of $L_{\alpha}\approx G_{{\alpha}_{1}
}\times\cdots\times G_{{\alpha}_{k}}$. We define
\[
\pi_{1}\times\cdots\times\pi_{k}=I_{P_{\alpha}}^{G_{n}}\left(  \pi_{1}\otimes\cdots\otimes\pi_{k}\right)
\]
\end{definition}

\RamiA{ $\pi_{1}\times\cdots\times\pi_{k}$ will be referred to below as the Bernstein-Zelevinsky product, or the BZ-product, or sometimes just the product of $\pi_{1},\dots,\pi_{k}$. It is well known (see e.g. \cite[\S 12.1]{Wal2}) that the product is commutative in the Grothendieck group.}

\subsection{Schwartz functions on Nash manifolds} \label{sec:ScNash}

In the proofs of Theorems \ref{thm:ExactHaus} and \ref{thm:Prod} we will use the language of Schwartz functions on Nash manifolds, as developed in \cite{AGSc}.

Nash manifolds are smooth semi-algebraic manifolds. In sections \ref{sec:ExactHaus} and \ref{sec:PfProd} only algebraic manifolds are considered and thus the reader can safely replace the word ``Nash" by ``real algebraic". However, in section \ref{sec:PfSmoothLem} we will use Nash manifolds which are not algebraic.

Schwartz functions are functions that decay, together with all
their derivatives, faster than any polynomial. On $\R^n$ it is the
usual notion of Schwartz function. We also need the notion of tempered functions, i.e. smooth functions that grow not faster than a polynomial,  and so do all their derivatives. For precise definitions of
those notions we refer the reader to \cite{AGSc}.
In this section we summarize some elements of the theory of Schwartz functions. We will give more details in \S \ref{subsec:PrelScNash}.

\begin{notation}
Let $X$ be a Nash manifold. Denote by $\Sc(X)$ the \Fre space of
Schwartz functions on $X$.
For any Nash vector bundle $\cE$ over $X$ we denote by $\Sc(X,\cE)$
the space of Schwartz sections of $\cE$.
\end{notation}

\begin{prop}[\cite{AGSc}, Theorem 5.4.3] \label{pOpenSet}
Let $U \subset X$  be a (semi-algebraic) open subset, then
$$\Sc(U,\cE) \cong \{\phi \in \Sc(X,\cE)| \quad \phi \text{ is 0 on } X
\setminus U \text{ with all derivatives} \}.$$
\end{prop}

\begin{notation}
Let $Z$ be a locally closed semi-algebraic subset of a Nash manifold $X$. Let $\cE$ be a Nash bundle over $X$. Denote $$\Sc_X(Z,\cE):=\Sc(X-(\overline{Z}-Z),\cE)/\Sc(X-\overline{Z},\cE).$$
Here we identify $\Sc(X-\overline{Z},\cE)$ with a closed subspace of $\Sc(X-(\overline{Z}-Z),\cE)$ using the description of
Schwartz functions on an open set (Proposition \ref{pOpenSet}).
\end{notation}

%\begin{cor} \label{cor:Sc_seq}
%Let $U \subset M$  be a (semi-algebraic) open subset and $Z:= X \setminus U$. Then
%we have a short exact sequence
%%
%$$ 0 \to \Sc(U, \cE) \to \Sc(X,\cE) \to \Sc_X(Z,\cE) \to 0.$$
%\end{cor}

\begin{cor}%[cf. \cite{AGS}, Appendix B]
\label{cor:Sc_seq}
Let $X:=\bigcup_{i=1}^k X_i$ be a Nash stratification of a Nash manifold $X$, i.e. $\bigcup_{i=j}^k X_i$ is an open Nash subset of $X$ for any $j$. Let $\cE$ be a Nash bundle over $X$.

Then $\Sc(X,\cE)$ has a natural filtration of length $k$ such that $Gr^i(\Sc(X,\cE)) = \Sc_X(X_i,\cE)$.

Moreover, if $Y$ is a Nash manifold and $X\subset Y$ is a (locally closed) Nash submanifold then $\Sc_Y(X,\cE)$ has a natural filtration of length $k$ such that $Gr^i(\Sc_Y(X,\cE)) = \Sc_Y(X_i,\cE)$.
\end{cor}

\begin{notation}\label{not:!}
Let $X$ be a Nash manifold.\\
(i) We denote by $D_X$ the Nash bundle of
densities on $X$. It is the natural bundle whose smooth sections
are smooth measures. For the precise definition see e.g.
\cite{AGRhamShap}.\\
(ii) Let $\phi:X \to Y$ be a map of (Nash) manifolds. Denote $D_Y^X:=D_\phi :=\phi^*(D_Y^*)\otimes D_X$.\\
(iii) Let $\cE \to Y$ be a (Nash) bundle. Denote $\phi^!(\cE)=\phi^*(\cE)\otimes D_Y^X$.
\end{notation}
\begin{rem}
If $\phi$ is a submersion then for all $y \in Y$ we have $D_Y^X|_{\phi^{-1}(y)} \cong D_{\phi^{-1}(y)}$.
%(ii) If a Lie group $G$ acts on a smooth manifold $X$ and $\cE$ is a $G$-equivariant vector bundle (i.e. we have a map $p^*(\cE) \to a^*(\cE)$, where $p:G \times X \to X$ is the projection and $a:G \times X \to X$ is the action) then we also have a natural map $p^!(\cE) \to a^!(\cE).$ If $G$, $X$ and $\cE$ are Nash and the actions of $G$ on $X$ and $\cE$ are Nash then the map $p^!(\cE) \to a^!(\cE)$ is Nash. If the action of $G$ on $\cE$ is tempered then the  map $p^!(\cE) \to a^!(\cE)$ corresponds to a tempered section of  $\Hom(p^!(\cE),a^!(\cE))$.
\end{rem}

\subsection{Tempered functions}\label{subsec:TempFun}

Analyzing smooth representations from a geometric point of view we encounter two related technical difficulties:
\begin{itemize}
\item Most characters of $F^\times$ are not Nash
\item $\psi$ is not Nash.
\end{itemize}
This makes the group action on the geometric objects that we consider to be not Nash. This could cause us to consider geometric objects outside the realm of Nash manifolds. In order to prevent that we introduce some technical notions, including the notion of $G$-tempered bundle which is, roughly speaking, a Nash bundle with a tempered $G$-action. The price of that is the need to make sure that all our constructions produce only Nash geometric objects.

\begin{notation}
Let $X$ be a Nash manifold. Denote by $\T(X)$ the space of tempered functions on $X$. For any Nash vector bundle $\cE$ over $X$ we denote by $\T(X,\cE)$ the space of tempered sections of $\cE$.
\end{notation}
Note that any unitary character of a real algebraic group $G$ is a tempered function on $G$.
%
%Since the character $\psi$ is not semi-algebraic, the generality of equivariant Nash bundles will not be wide enough for us. Thus we will need the following notions.

\begin{defn}
Let $X$ be a Nash manifold and let $\cE$ and $\cE'$ be Nash bundles over it.
A morphism of bundles $\phi:\cE \to \cE'$ is called tempered if it
corresponds to a tempered section of the bundle $\Hom(\cE, \cE')$.
\end{defn}
Note that if $\phi:\cE \to \cE'$ is a tempered morphism of bundles, then the induced map $ C^{\infty}(X,\cE) \to C^{\infty}(X,\cE')$ maps $\Sc(X,\cE)$ to $\Sc(X,\cE')$ and $\T(X,\cE)$ to $\T(X,\cE')$.

\begin{defn}
Let a Nash group $G$ act on a Nash manifold $X$. A {\bf $G$-tempered bundle} $\cE$ over $X$ is a Nash bundle $\cE$ with an equivariant structure $\phi:a^*(\cE) \to p^*(\cE)$ (here $a:G\times X\to X$ is the action map and  $p:G\times X\to X$ is the projection) that is a tempered morphism of bundles such that for any element $L$ in the Lie algebra of $G$ and any open (semi-algebraic) subset $U \subset X$ the derived map $a(L):C^\infty(U,\cE) \to C^\infty(U,\cE)$ preserves the sub-space of Nash sections of $\cE$ on $U$.
\end{defn}

\begin{defn}\label{def:MultBun}$\,$
\begin{itemize}
\item We call a character $\chi$ of a Nash group $G$ \emph{multiplicative} if $\chi = \mu \circ \overline{\chi}$, where $\overline{\chi}:G \to F^\times$ is a homomorphism of Nash groups and $\mu:F^\times \to \C^\times$ is a character. Note that all multiplicative characters are tempered.
\item A multiplicative representation of $G$ is a product of a (finite-dimensional) Nash representation of $G$ with a multiplicative character.
\item Let a Nash group $G$ act on a Nash manifold $X$. A $G$-tempered bundle $\cE$ over $X$ is called {\bf $G$-multiplicative bundle} if for any point $x \in X$, the fiber $\cE_x$ is a multiplicative representation of the stabilizer $G_x$.
\end{itemize}
\end{defn}

\section{Proof of exactness and Hausdorffness (Theorem \ref{thm:ExactHaus})} \label{sec:ExactHaus}
\setcounter{lemma}{0}

We will prove Theorem  \ref{thm:ExactHaus} for the ``untwisted" functor $\oE^k$. Clearly, these versions are equivalent.

%In order to prove the theorem we will prove the following theorem.
The proof %of Theorem  \ref{thm:ExactHaus}
is based on the following theorem.
\begin{thm}\label{thm:Class}
There exists a sequence of collections $\{C_n\}_{n=1}^{\infty}$ of topological Hausdorff $\fp_n$-representations such that

\begin{enumerate}
\item Any representation induced from a finite-dimensional representation of the Borel subgroup lies in $C_n$.
%The smooth principal series representations of $G_n$ are in $C_n$.
\item For any $\pi \in C_n$, $\pi \otimes (-\psi)$ is $\fv_n$-acyclic (as a linear representation).
\item For any $\pi \in C_n, \text{ we have } \pi_{\fv_n,\psi} \in C_{n-1}.$
\end{enumerate}

\end{thm}
We will prove this theorem in \S \ref{subsec:PfClass}.
\begin{cor}\label{cor:Class0}
Any $\fp_n$-representation $\pi$ of the class $C_n$  is acyclic with respect to $\oE^k$ and $\oE^k(\pi)$ is Hausdorff, for any $0<k\leq n$.
\end{cor}

This corollary follows from Theorem \ref{thm:Class}, using the following lemma that we will prove in \S \ref{subsec:PfCompD1Acyc}.
\begin{lem}\label{lem:CompD1Acyc}
Let $\pi$ be a representation of $\fp_n$. Suppose that $\pi$ is $\oPhi$-acyclic and $\oPhi(\pi)$ is $\oE^k$-acyclic. Then $\pi$ is $\oE^{k+1}$-acyclic.
\end{lem}

\begin{cor}\label{cor:Class}
%Any principal series representation $\pi$ of $G_n$
Any representation $\pi$ of $G_n$  induced from a finite-dimensional representation of the Borel subgroup is acyclic with respect to $\oE^k$ and $\oE^k(\pi)$ is Hausdorff, for any $0<k\leq n$. %and a principal series representation .
\end{cor}

Now we are ready to deduce Theorem \ref{thm:ExactHaus}. The deduction follows the lines of the proof of \cite[Lemma 8.4]{CHM}.
\begin{proof}[Proof of Theorem \ref{thm:ExactHaus}]
$\,$
\begin{itemize}
\item [(\ref{it:Exact})]
Using the long exact sequence of Lie algebra homology (see Lemma \ref{lem:LieHom} below)  it is enough to show that for any $\pi \in \cM(G_n)$, the $\fu_n^{k}$-representation $\pi \otimes (-\psi_{n}^k)$ is acyclic.

We will prove by downward induction on $l$ that $H_l(\fu_n^{k},\pi \otimes (-\psi_{n}^k))=0$ for any $\pi \in \cM(G_n)$.
For $l > \dim \fu_{n}^k$ it follows from the Koszul complex (see \S \ref{subsubsec:LieAlgCoh}). Let us assume that the statement holds for $l+1$ and prove it for $l$.   By the Casselman subrepresentation theorem (Theorem \ref{thm:CasSubRep}), one can embed $\pi$ into a
representation $I$  induced from a finite dimensional representation of the Borel subgroup.
Consider the short exact sequence
$$ 0 \to \pi \to I \to I/\pi \to 0.$$
By Theorem \ref{cor:CW}%the Casselman-Wallach theorem (Theorem \ref{thm:CW})
, $I/\pi \in \cM(G_n)$ and thus, by the induction hypothesis, $H_{l+1}(\fu_{n}^k,(I/\pi) \otimes (-\psi_{n}^k))=0$.
By Corollary \ref{cor:Class}, $H_{i}(\fu_{n}^{k},I \otimes (-\psi_{n}^k))=0$ for all $i>0$. Thus,  $H_{l}(\fu_{n}^{k},\pi \otimes (-\psi_{n}^k))=H_{l+1}(\fu_{n}^{k},(I/\pi) \otimes (-\psi_{n}^k))=0$.

\item [(\ref{it:Haus})]  By the Casselman subrepresentation theorem (Theorem \ref{thm:CasSubRep}),  one can embed $\pi$ into a representation $I$  induced from a finite dimensional representation of the Borel subgroup. By (\ref{it:Exact}), we have an embedding $\oE^k(\pi) \hookrightarrow \oE^k(I)$. By Corollary \ref{cor:Class}, $\oE^k(I)$ is Hausdorff and hence $\oE^k(\pi)$ is Hausdorff.
\end{itemize}
\end{proof}

\subsection{Good $\fp_n$-representations and the key lemma}\label{subsec:PfClass}

%\begin{lem}
%Let $X$ be a Nash manifold with a Nash $V_n$-action. Suppose that there exists a geometric quotient $X'=X/V_n$, i.e. a Nash manifold $X'$ and a Nash (locally trivial) fibration $p_X':X \to X'$ such that the fibers are $V_n$-orbits. Then there exists a unique Nash bundle $\cE_{X,V}$ on $X'$ such that $p^!(\cE)$ is isomorphic to the trivial bundle.
%\end{lem}
%
%\begin{notn}
%In the situation of the previous lemma, the bundle is denoted $\cE_{X,V}$.
%\end{notn}
The following lemma will play a key role in this and the next section.

%?? search for key, erase "analytic"

\begin{lem}[The key lemma]\label{lem:GeoGood}
Let $T$ be a Nash linear group, i.e. a Nash subgroup of $GL_k(\bR)$ for some $k$, and let $R:=P_n \times T$ and $R':=G_{n-1}\times T$.
Let $Q < R$ be a Nash subgroup and let $X:=R/Q$ and $X':=V_n \backslash X$.
Note that $X' =R'/Q',$ where $Q'=Q/(Q \cap V_n)$.
%Let $X$ be a Nash manifold with a Nash $P_n$-action. Suppose that there exists a geometric quotient $X'=X/V_n$.
%, i.e. a Nash manifold $X'$ and a Nash (locally trivial) fibration $X \to X'$ such that the fibers are $V_n$-orbits.
%$Y=X/P_n$, i.e. a Nash manifold $Y$ and a locally trivial fibration $X \to Y$ such that the fibers are $P_n$-orbits. Note that this implies that there is a geometric quotient $X'=X/V_n$.
Let $X_0= \{ x \in X \, : \, \psi|_{(V_n)_x}=1\}.$ Let $X_0'$ be the image of $X_0$ in $X'$.  Let $\cE'$ be an $R'$-tempered bundle on $X'$ and $\cE:=p_{X'}^!(\cE')$ be its pullback to $X$.
Then
\begin{enumerate}[(i)]
\item $\oH_i(\fv_n,\Sc(X,\cE)\otimes (-\psi))=0$ for any $i>0$.
\item $X_0'$ is smooth
\item
%Let $r:\Sc(X,\cE) \to \Sc(X_0,\cE|_{X_0})$ denote the restriction. Then ${p_X'}_* \circ r$ gives an isomorphism
$\oH_0(\fv_n,\Sc(X,\cE)\otimes (-\psi))\cong \Sc(X_0',\cE|_{X_0'})$ as representations of $\fp_{n-1}$.
\end{enumerate}
\end{lem}
We will prove this lemma in \S\ref{sec:PfGeoGood}.

\begin{defn}
Let $B_n<G_n$ denote the standard (upper-triangular) Borel subgroup.
In the situation of the above lemma, assuming that $(B_{n}\cap P_n) \times T$ has finite number of orbits on $X$, and that $\cE'$ is $R'$-multiplicative we call the representation $\Sc(X,\cE)$ a geometric representation of $\fp_n$.
\end{defn}

\begin{defn}
We define a good extension of nuclear \Fre spaces to be the following data:
\begin{enumerate}
\item a nuclear \Fre space $W$
\item a countable descending filtration $F^i(W)$ by closed subspaces
\item a sequence of nuclear \Fre spaces $W_i$
\item isomorphisms $\phi_i:F^{i+1}(W)/F^i(W) \to W_i$
\end{enumerate}
such that the natural map $W \to \lim \limits _\ot W/F^i(W)$ is an  isomorphism of (non-topological) linear spaces.

In this situation we  will also say that $W$ is a good extension of $\{W_i\}$. If a Lie algebra $\g$ acts on $W$ and on each $W_i$, preserving each $F^i(W)$ and commuting with  each $\phi_i$ we will say that this good extension is $\g$-invariant.
\end{defn}

\begin{defn}
We define the collection of good representations of $\fp_n$ to be the  smallest collection that includes the geometric representations and is closed with respect to
$\fp_n$-invariant good extensions.
 %and completed tensor products by  nuclear \Fre spaces (with trivial action of $\fp_n$).
\end{defn}
We will prove the following concretization of Theorem \ref{thm:Class}:

\begin{thm}\label{thm:ClassCon}
%The collections of good representations satisfy %$\{C_n\}_{n=1}^{\infty}$ of topological Hausdorff $\fp_n$-representations such that
$ $
\begin{enumerate}
\item \label{it:PrinGood} The representations of $G_n$  induced from a finite dimensional representation of the Borel subgroup are  good representations of $\p_n$.
\item \label{it:GoodAcyc} For any good representation $\pi$ of $\p_n$, $\pi \otimes (-\psi)$ is $\fv_n$-acyclic (as a linear representation).
\item \label{it:GoodGood} For any good representation $\pi$ of $\p_n$, the representation $\pi_{\fv_n,\psi}$    is good representation of $\p_{n-1}$.
\end{enumerate}
\end{thm}

For the proof we will need some lemmas.
\begin{lem}\label{lem:GenGeoGood}
Let $T$ be a Nash linear group, and let $R:=P_n \times T$ and $R':=G_{n-1}\times T$.
Let $Y$ be an $R$-Nash manifold.
Let $X \subset Y$ be an $R$-invariant (locally closed) Nash submanifold such that $(B_n\cap P_n) \times T$ has finite number of orbits on $X$.
Let $\cE$ be an $R$-multiplicative bundle on $Y$.
Then the representation  $\Sc_Y(X,\cE)$ is good.
\end{lem}
We will prove this lemma in \S \ref{subsec:PfGenGeoGood}.

\begin{lem}\label{lem:PrinGood}
Let $G$ be a real algebraic group and $H$ be a Zariski closed subgroup.
Let $\chi$ be a multiplicative character of $H$. Let $K$ be a Zariski closed subgroup of $G$ such that $KH=G$ and $\chi|_{K\cap H}$ is an algebraic character.
Let $\cE=G \times_H \chi$ be the smooth bundle on $G/H$ obtained by inducing the character $\chi$. Then $\cE$ admits a structure of a $G$-multiplicative bundle.
\end{lem}
We will prove this lemma in \S \ref{subsec:PfPrinGood}.

\begin{cor}\label{cor:PrinGood}
Let $\chi$ be a character of the torus $T_n <B_n<G_n$, continued trivially to $B_n$. Let $\pi = Ind_{B_n}^{G_n}(\chi)$.
Then there exists a $G_n$-multiplicative bundle $\cE$ on $G_n / B_n$ such that $\pi = \Sc(G_n / B_n,\cE)$.
\end{cor}

\begin{lem}\label{lem:GoodExtExactHaus}
Let $\g$ be a Lie algebra.
Let $(\pi,F,\{\pi_i\})$ be a $\g$-invariant good extension. Suppose that $\pi_i$ are $\g$-acyclic and $\oH_0(\g,\pi_i)$ is Hausdorff for any $i$. Then $\pi$ is acyclic and $\oH_0(\g,\pi)$ together with the induced filtration $F^i(\oH_0(\g,\pi))= F^i(\pi)/(\g \pi\cap F^i(\pi))$ and the natural morphisms $ F^i(\oH_0(\g,\pi))/F^{i+1}(\oH_0(\g,\pi))  \to \oH_0(\g,\pi_i)$ form a good extension. In particular, $\oH_0(\g,\pi)=\pi_{\g}$ is Hausdorff.
\end{lem}

We will prove this lemma in \S \ref{subsec:GoodExtExactHaus}.

\begin{proof}[Proof of Theorem \ref{thm:ClassCon}]
$ $
\begin{enumerate}
\item [(\ref{it:PrinGood})]

Let $\sigma$ be a finite dimensional representation of the Borel subgroup and let $I \in \cM(G_n)$ be its induction to $G_n$.
There exists be a filtration of $\sigma$ such that $Gr^i(\sigma)$ is a character of the Borel subgroup and thus of the torus. This filtration induces a filtration of $\pi$. It is enough to show that $Gr^i(I)$ is good. Note that $Gr^i(I)$ is a principal series representation and thus it is good
%Let $\pi$ be a principal series representation. By definition, it is induced %from a finite dimensional representation of the torus.
%Without loss of generality, by the definition of goodness, we can assume %that $\sigma$ is a character. In this case $\pi$ is good
by Corollary \ref{cor:PrinGood} and Lemma \ref{lem:GenGeoGood}.

\item [(\ref{it:GoodAcyc})] The key lemma (Lemma \ref{lem:GeoGood}) implies that $\pi \otimes (-\psi)$ is acyclic for any geometric representation $\pi$. Hence Lemma \ref{lem:GoodExtExactHaus}
    %and \ref{lem:TenFreEx}
    implies by induction that $\pi \otimes (-\psi)$ is acyclic for any good representation $\pi$.
\item [(\ref{it:GoodGood})] If $\pi = \Sc(X,\cE)$ is a geometric representation then by the key lemma (Lemma \ref{lem:GeoGood}), $\pi_{\fv_n, \psi} \cong \Sc(X'_0, \cE'|_{X'_0}),$   where $\cE'$ and $X'_0$ are as in the key lemma. %Lemma \ref{lem:GeoGood}.
    Let $\widetilde{T}$ denote the product of $T$ with the center of $G_{n-1}$.
    Clearly, $\widetilde{T}$ acts on $X'_0$ and $B_{n-1} \times T = (B_{n-1} \cap P_{n-1}) \times \widetilde{T}$ has a finite number of orbits on $X'_0$.
    Thus, by Lemma  \ref{lem:GenGeoGood}, $\pi_{\fv_n, \psi}$ is good.
    %    If $\pi_{\fv_n, \psi}$ is good and $\cE$ is a nuclear \Fre space then $(\pi \hot \cE)_{\fv_n, \psi}=\pi_{\fv_n, \psi} \hot \cE$ is good.

    Now let $\pi$ be a good representation. We can assume by induction that $\pi$ is a good extension of $\pi_i$, where $\pi_i$ are good and $(\pi_i)_{\fv_n, \psi}$ are good. Then Lemma \ref{lem:GoodExtExactHaus} implies that $\pi_{\fv_n, \psi}$ is a good extension of $(\pi_i)_{\fv_n, \psi}$ and hence is good.
\end{enumerate}
\end{proof}

For the proof of Theorem \ref{thm:Prod} we will need the following corollary of Lemma \ref{lem:GoodExtExactHaus} and Theorem \ref{thm:ClassCon}.
\begin{cor}\label{cor:ExactPf}
%Let $\g$ be a Lie algebra.
Let $(\pi,F,\pi_i)$ be a $\p_n$-invariant good extension of good representations of $\p_n$. %Suppose that $\pi_i$ are $\g$-acyclic and $\oH_0(\pi_i)$ is Hausdorff %for any $i$.
Then %$\pi$ is acyclic and
$E^k(\pi)$ together with the induced filtration defined by the images of  $E^k(F^i(\pi))$ and the natural morphisms $ F^i(E^k(\pi))/F^{i+1}(E^k(\pi))  \to E^k(\pi_i)$ form a good extension.
\end{cor}

\section{Highest derivative of monomial representations (Proof of Theorem \ref{thm:Prod})}\label{sec:PfProd}
\setcounter{lemma}{0}

We first sketch the proof for the case of product of two characters. We believe that using an appropriate notion of \Fre bundles it will be possible to upgrade this proof to the case of the highest derivative of a product of two arbitrary representations, which by induction will give a more straightforward proof of Theorem \ref{thm:Prod}.

Since we do not currently have a proof for a product of two representations, in \S \ref{subsec:ProdGeo} - \ref{subsec:PfProd} we prove Theorem \ref{thm:Prod} directly for a product of $k$ characters.

\subsection{Sketch of proof for the case $k=2$} \label{subsec:PfProd2}

First, note that there exists a $G_n$-multiplicative line bundle $\cE$ over $G_n/P_{(n_1,n_2)}$ such that $\chi_1 \times \chi_2 = \Sc(G_n/P_{(n_1,n_2)},\cE)$ and the action of $P_{(n_1,n_2)}$ on the fiber is given by trivial extension of $\chi_1 \otimes \chi_2$, twisted by a power of the determinant.

Let $w_n \in G_n$ denote the longest Weyl group element. Let $x_1, x_2 \in X=G_n/P_{(n_1,n_2)} $ be the classes of $1$ and $w_n$ in correspondence and let $\cO_1$ and $\cO_2$ be their $P_n$-orbits. Note that $X = \cO_1 \cup \cO_2$. Thus, by Corollary \ref{cor:Sc_seq}, we have a short exact sequence
$$ 0 \to \Sc(\cO_2,\cE) \to \Sc(X,\cE) \to \Sc_X(\cO_1,\cE) \to 0.$$
By Lemma \ref{lem:GenGeoGood}, $\Sc(O_2,\cE)$ and $\Sc_X(O_1,\cE)$ are good. By Theorem \ref{thm:ClassCon},
$\Sc(\cO_2,\cE)$ and $\Sc_X(\cO_1,\cE)$ are acyclic with respect to $\Phi$ and thus we have a short exact sequence
$$ 0 \to \Phi(\Sc(\cO_2,\cE)) \to \Phi(\Sc(X,\cE)) \to \Phi(\Sc_X(\cO_1,\cE)) \to 0.$$
The proposition follows from the following 2 statements:
\begin{enumerate}
\item \label{it:Phi0} $\Phi(\Sc_X(\cO_1,\cE)) =0$
\item \label{it:PhiProd} $\Phi(\Sc(\cO_2,\cE)) = (\pi_1)|_{G_{n_1-1}} \times (\pi_2)|_{G_{n_2-1}}$.
\end{enumerate}

Let us first show (\ref{it:Phi0}).
Using the results of \S \ref{sec:ExactHaus} and a version of the Borel lemma (Lemma \ref{lem:Borel}),
we reduce to the statement $\Phi(\Sc(\cO_1,\cE'))=0$ for any $P_n$-multiplicative equivariant bundle $\cE'$.
Let $p: \cO_1 \to V_n\backslash \cO_1$ be the natural projection.
Using the nilpotency of the action of $V_n$ on $\cE'$ we reduce to the case $\cE' = p^!(\cE'')$. This case follows by the key lemma (Lemma \ref{lem:GeoGood}) from the inclusion $(P_n)_{x_1}=P_n \cap P_{(n_1,n_2)} \supset V_n$.

The proof of \eqref{it:PhiProd} is a computation based on the key lemma (Lemma \ref{lem:GeoGood}).

\subsection{Structure of the proof}\label{subsec:ProdStruc}

%We study the degenerate principal series geometrically.
Let $\lambda=(n_1 , \dots ,  n_k)$ be a composition of $n$. The monomial representation is isomorphic to $\Sc(G_n/P_{\lam},\cE)$ for some $G_n$-multiplicative line bundle $\cE$ over $X:=G_n/P_{\lam}$. Let $\cO_1 , \dots ,  \cO_k$ be the $P_{n}$-orbits on $G_n/P_{\lam}$, where $\cO_1$ is the open orbit. We reduce the problem to the following two tasks:

\begin{itemize}
\item to compute $E^k(\Sc(\cO_1,\cE))$
\item to show that $E^k(\Sc_X(\cO_i,\cE'))=0$ for any $P_n$-multiplicative equivariant bundle $\cE'$ on $X$ and any $2\leq i\leq k$.
\end{itemize}
We do both tasks by induction using the  key lemma (Lemma \ref{lem:GeoGood}) and the fact that the $P_{n-1}$-orbits on the geometric quotient $V_n\backslash \cO_i$ looks exactly like $\cO_j$ when $n$ is replaced by $n-1$.

In \S \ref{subsec:ProdGeo} we study the geometry of $P_n$-orbits on $X$. In \S \ref{subsec:ProdDer} we reformulate the above tasks in explicit lemmas, that will be proven in \S \ref{subsec:PfLemProd}. In \S \ref{subsec:PfProd} we prove Theorem \ref{thm:Prod} using those lemmas.
\subsection{Geometry of $P_n$-orbits on flag varieties} \label{subsec:ProdGeo}
Let $\lambda=(n_1 , \dots ,  n_k)$ be a composition of $n$.
In this subsection we describe the orbits of $P_n$ on $G_n/P_{\lam}$. Note that the scalar matrices act trivially on $G/P_{\lam}$ and thus $P_n$ orbits coincide with $P_{(n-1,1)}$-orbits.
By Bruhat theory, $P_{(n-1,1)} \backslash G_n / P_{\lam}$ can be identified with the set of double cosets of the symmetric group permuting the set $\{1 , \dots ,  n\}$ by subgroups corresponding to $P_{(n-1,1)}$ and $P_{\lam}$ respectively. The first subgroup is the stabilizer of the point $n$, and the second is the subgroup of all permutations that preserve the segment $[n_{i-1}+1,n_i]$ for each $i$. Thus, the double quotient consists of $k$ elements. Now we want to find a representative in $G_n$ for each double coset. We can choose all of them to be powers of the standard cyclic permutation matrix.
This discussion is formalized in the following notation and lemma.

\begin{notn}\label{not:Orbits1}$ $
%Let $\lambda=(n_1 , \dots ,  n_k)$ be a composition of $n$.
\begin{itemize}
\item Let $c\in G_n$ denote the standard cyclic permutation matrix, i.e. $ce_i=e_{i+1}$ for $i<n$ and $ce_{n}=e_1$, where $e_i$ are basis vectors.
\item $m^i_{\lam}:=\sum_{l=1}^i n_l$. In particular, $m_{\lam}^k=n$.
\item For $1 \leq i \leq k,$ let $w^i_\lambda:=c^{n-m_{\lam}^{k-i+1}} \in G_n$
\item $P_{\lambda}^i:=w^i_\lambda P_{\lambda} (w^i_\lambda)^{-1}.$
\item Let $x_{\lambda}^i$ be the class of $w^i_\lambda$ in $G_n / P_{\lambda}$ and $\cO_{\lambda}^i:=P_{n}x_{\lambda}^i$.
\item $Q_{\lambda}^i:=P_{\lambda}^i \cap P_n$.
\end{itemize}
\end{notn}

See Appendix \ref{app:ProdEx} for examples of the objects described in this and the following notations.

\begin{lem}\label{lem:SysRep}
$ $
\begin{enumerate}[(i)]
\item The stabilizer of $x_{\lambda}^i$ is $Q_{\lambda}^i$.
\item $\cO_{\lambda}^i$ are all distinct.
\item $G_n / P_{\lambda} = \bigcup_{i=1}^k \cO_{\lambda}^i$
\item For any $l \leq k$, $\bigcup_{i=1}^l \cO_{\lambda}^i$ is closed in $G_n / P_{\lambda}$ .
\end{enumerate}
\end{lem}

We will also be interested in the representatives of $P_n$-orbits on $G_n/P_{\lambda}^i$.
\begin{notn}
Let $\lambda=(n_1 , \dots ,  n_k)$ be a composition of $n$. For $1 \leq i,j \leq k,$
\begin{itemize}
\item Let $w^{ij}_\lambda:= (w^{i}_\lambda)^{-1} w^{j}_\lambda\in G_n$
\item Let $x_{\lambda}^{ij}$ be the class of $w^{ij}_\lambda$ in $G_n / P_{\lambda}^i$ and $\cO_{\lambda}^{ij}:=P_{n}x_{\lambda}^{ij}$.
\end{itemize}
\end{notn}

\begin{cor}\label{cor:SysRep}
Fix $i \leq k$. Then
\begin{enumerate}[(i)]
\item The stabilizer of $x_{\lambda}^{ij}$ is $Q_{\lambda}^j$.
\item $\cO_{\lambda}^{ij}$ are all distinct.
\item $G_n / P_{\lambda}^i = \bigcup_{j=1}^k \cO_{\lambda}^{ij}$
\item For any $l \leq k$, $\bigcup_{j=1}^l \cO_{\lambda}^{ij}$ is closed in $G_n / P_{\lambda}^i $.
\end{enumerate}

$\{x_{\lambda}^{ij}\}_{j=1}^k$ is a full system of representatives for $P_n$-orbits in $G_n / P_{\lambda}^i$, and the stabilizer of $x_{\lambda}^{ij}$ is $Q_{\lambda}^j$.
\end{cor}

\begin{notn}
For any $1 \leq i \leq k$, denote $\lambda^-_{i}:=(n_1 , \dots ,  n_{i-1},n_{i}-1,n_{i+1} , \dots ,  n_k)$.
\end{notn}

The following lemma is a straightforward computation (see Appendix \ref{app:ProdEx} for a particular example.).
\begin{lem}\label{lem:ProdGeo}
Let $\lambda=(n_1 , \dots ,  n_k)$ be a composition of $n$.
Let $X:=P_n/Q_{\lambda}^i$ and $Z:=(V_n)\backslash X$ be the geometric quotient.
Note that $Z=G_{n-1}/P^i_{\lambda'}$, where $\lambda'=\lambda^-_{k-i+1}$.
%$\lambda'=(n_1 , \dots ,  n_{k-i},n_{k-i+1}-1,n_{k-i+2} , \dots ,  n_k)$.

Let $L=Q_{\lam}^i \cap V_n$ and $\widetilde{Z}_0:=\{g \in G_{n-1} \, : \, \psi(gL) =1\}$. Note that $\widetilde{Z}_0$ is right-invariant with respect to $P_{\lam'}^i$.
Let $Z_0:=\widetilde{Z}_0/P^i_{\lam '} \subset Z$.

Then $x_{\lam '}^{ij} \in Z_0$ if and only if $1 \leq j < i$.
\end{lem}

\subsection{Derivatives of quasi-regular representations on $P_n$-orbits on flag varieties}\label{subsec:ProdDer}

The proof of Theorem \ref{thm:Prod} is based on two lemmas that we formulate in this subsection and prove in \S \ref{subsec:PfLemProd}.

\begin{lem}\label{lem:ProdDer0}
Let $\lambda=(n_1 , \dots ,  n_k)$ be a composition of $n$.
Let $Y$ be a Nash $P_n$-manifold and $\cE$ be a $P_n$-multiplicative bundle on $Y$.
Let $x_0 \in Y$ be a point with stabilizer $Q_{\lambda}^i$ and let $X:=P_nx_0$.
Let $\pi:=\Sc_Y(X,\cE)$. Then
$E^{i+1}(\pi)=0 $.
\end{lem}

\begin{notn}\label{not:Orbits2}
$\,$
\begin{itemize}
\item Denote by $\Chi$ the set of all characters of $F^\times$.
\item Let $\alp = (\alp_1 , \dots ,  \alp_k) \in \Chi^k$. Denote by
$\overline{\xi}_{\lam,\alp}$ the character of $L_{\lam}$ defined by
$$\overline{\xi}_{\lam,\alp}(g_1 , \dots ,  g_k)= \prod_{i=1}^k(\alp_i(\det(g_i)).$$
and by $\xi_{\lam,\alp}$ the character of $L_{\lam}$ defined by
$$\xi_{\lam,\alp}(g_1 , \dots ,  g_k)= \overline{\xi}_{\lam,\alp}  \RamiA{\Delta} _{P_\lam}^{1/2}|_{L_{\lam}}=\prod_{i=1}^k(\alp_i(\det(g_i))|\det (g_i)|^{(n-m_{\lam}^i-m_{\lam}^{i-1})/2}).$$

\item Let $\chi_{\lam,\alp}$ and $\overline{\chi}_{\lam,\alp}$  be the extensions of $\xi_{\lam,\alp}$ and $\overline{\xi}_{\lam,\alp}$ to $P_{\lam}$.
\item Denote by $\chi_{\lam,\alp}^i$ and $\overline{\chi}_{\lam,\alp}^i$ the characters of $P_{\lam}^i$ defined by $\chi_{\lam,\alp}^i(g)=\chi_{\lam,\alp}((w_\lam^i)^{-1}gw_\lam^i)$ and $\overline{\chi}_{\lam,\alp}^i(g)=\overline{\chi}_{\lam,\alp}((w_\lam^i)^{-1}gw_\lam^i)$.
%\item Let $\beta_{k}^i:=(|\det|^{-1/2} , \dots ,  |\det|^{-1/2},1,|\det|^{1/2} , \dots ,  |\det|^{1/2}) \in \Chi^k$

\end{itemize}
\end{notn}

\begin{lem}\label{lem:ProdLastDer}
Let $\lambda=(n_1 , \dots ,  n_k)$ be a composition of $n$.
Let $X:=P_n/Q_{\lambda}^i$. Let $\cE$ be a $P_n$-multiplicative line bundle on $X$. Let $\pi:=\Sc(X,\cE)$. Let $x_0 \in X$ be the class of the trivial element of $P_n$. Consider $\cE|_{x_0}$ as a character of $Q_{\lambda}^i$. Suppose that this character is the restriction of $\chi_{\lam,\alp}^i$ to $Q_{\lambda}^i$, for some $\alp=(\alp_1 , \dots ,  \alp_k) \in \Chi^k$. Then
\DimaA{
$$E^{i}(\pi)|_{G_{n-i}} = \chi_{(n_1),(\alp_1 |\det|^{-1/2})} \times  \RamiA{\cdots}  \times\chi_{(n_{k-i}),(\alp_{k-i}|\det|^{-1/2})} \times \chi_{(n_{k-i+1}-1),(\alp_{k-i+1})} \times  \RamiA{\cdots}  \times \chi_{(n_k-1),(\alp_k)}$$}
%$$\Phi^{i-1}(\pi)|_{G_{n-i}} = \chi_{(n_1),(\alp_1 |\det|^{-1/2})} \times  \RamiA{\cdots}  \times\chi_{(n_{k-i}),(\alp_{k-i}|\det|^{-1/2})} \times \chi_{(n_{k-i+1}-1),(\alp_{k-i+1})} \times  \RamiA{\cdots}  \times \chi_{(n_k-1),(\alp_k)}$$
\end{lem}

\subsection{Proof of Theorem \ref{thm:Prod}}\label{subsec:PfProd}

\begin{lem}\label{lem:ExBun}
Let $\alp=(\alp_1 , \dots ,  \alp_k)\in \Chi^k$.
There exists a $G_n$-multiplicative line bundle $\cE$ over $G/P_{\lam}$ such that $\chi_{(n_1),(\alp_1)} \times \RamiA{\cdots}  \times \chi_{(n_k),(\alp_k)} = \Sc(G/P_{\lam},\cE)$ and the action of $P_{\lam}$ on the fiber $\cE_{x_{\lam}^1}$ is given by the character $\chi_{\lam,\alp}$.
\end{lem}
This lemma follows from Lemma \ref{lem:PrinGood}.

\begin{proof}[Proof of Theorem \ref{thm:Prod}]
Let $X:=G_n/P_{\lam}$ and let $\cE$ be the $G_n$-multiplicative equivariant bundle given by Lemma \ref{lem:ExBun}. We have to show that $E^k(\Sc(X,\cE))=\chi_{(n_1-1),(\alp_1)} \times \RamiA{\cdots}  \times \chi_{(n_k-1),(\alp_k)}$. Recall that by Lemma \ref{lem:SysRep},
$G_n / P_{\lambda} = \bigcup_{i=1}^k \cO_{\lambda}^i$.
%\item The stabilizer of $x_{\lambda}^i$ is $Q_{\lambda}^i$.
%\item $\cO_{\lambda}^i$ are all distinct.
%\item $G_n / P_{\lambda} = \bigcup_{i=1}^k \cO_{\lambda}^i$
%\item For any $l \leq k$, $\bigcup_{i=1}^l \cO_{\lambda}^i$ is closed in $G_n / P_{\lambda}$ .
Thus, by Corollary \ref{cor:Sc_seq} there is a finite filtration on $\Sc(X,\cE)$ such that $Gr^i(\Sc(X,\cE)) = \Sc_X(\cO_{\lambda}^i,\cE)$. By Lemma \ref{lem:GenGeoGood}, each $\Sc_X(\cO_{\lambda}^i,\cE)$ is a good representation of $\fp_n$ and thus, by Theorem \ref{thm:ClassCon}, $\Sc_X(\cO_{\lambda}^i,\cE)$ is $E^k$-acyclic. Thus, $E^k(\Sc(X,\cE))$ has a filtration such that $Gr^i(E^k(\Sc(X,\cE))) = E^k(\Sc_X(\cO_{\lambda}^i,\cE)).$
By Lemma \ref{lem:ProdDer0}, $E^k(\Sc_X(\cO_{\lambda}^i,\cE))=0$ for all $i<k$. By Lemma \ref{lem:ProdLastDer}, $$E^{k}(\Sc(\cO_{\lambda}^i,\cE))|_{G_{n-k}} = \chi_{(n_1-1),(\alp_1)} \times  \RamiA{\cdots}  \times \chi_{(n_k-1),(\alp_k)}$$
Thus,
%\begin{multline}
$$
E^k(\Sc(X,\cE))|_{G_{n-k}} = E^k(\Sc(\cO_{\lambda}^k,\cE))|_{G_{n-k}}= \chi_{(n_1-1),(\alp_1)} \times  \RamiA{\cdots}  \times \chi_{(n_k-1),(\alp_k)}
$$
%(\Phi^{k-1}(\Sc(\cO_{\lambda}^k,\cE)))|_{G_{n-k}} =\\ \chi_{(n_1-1),(\alp_1|\det|^{-1/2})} \times  \RamiA{\cdots}  \times \chi_{(n_k-1),(\alp_k|\det|^{-1/2})}
%\end{multline}

\end{proof}

\begin{rem}
Note that this proof gives a description of $E^l(\chi_1 \times  \RamiA{\cdots}  \times \chi_k)$ for an arbitrary $l$. More precisely, we get a filtration on this space and a description of each quotient. However, this is an infinite filtration, indexed by a rather complicated ordered set.

In fact, we have a recipe to compute $\Phi(\Sc(P_{n}/Q_{\lambda}^i,\cE))$ for any multiplicative bundle $\cE$. First, we use Lemma \ref{lem:TranBun} to reduce to the case when $\cE$ is the pullback of a bundle $\cE'$ on the geometric quotient $P_n/Q_{\lambda}^i V_n$. Then the  key lemma (Lemma \RamiA{\ref{lem:GeoGood}}) says $\Phi(\Sc(P_{n}/Q_{\lambda}^i,\cE))=\Sc(X_0',\cE'),$ where $X_0' \subset P_n/Q_{\lambda}^i V_n$. In this case the set $X_0'$ coincides with the set $Z_0$ that is explicitly described in Lemma \ref{lem:ProdGeo}.
This gives us a filtration on $\Sc(X_0',\cE')$ such that the associated graded pieces are of the type $\Sc(P_{n-1}/ Q^j_{\lambda^-_i},\cE_{jk})$ where $1 \leq j <i$ and $k \geq 0$.

Since the orbits of $P_n$ on $X/P_{\lambda}$ are $P_n/Q_{\lambda}^i$, we can use this recipe in order to describe the value of the functor $\Phi$ on products of finite-dimensional representations and then, proceeding inductively, we can describe the functor $E^i$ on such representations.
\end{rem}

\section{Proofs of some technical lemmas}\label{sec:PfSmoothLem}

%\subsection{Finite dimensional dense subspaces (Proof of Lemma \ref{lem:FinDense})} \label{subsec:PfFinDense}
%
%
%%For the proof we will need the following two statements.
%
%\begin{lem}
%Any locally convex Hausdorff topology on a finite dimensional vector space is equivalent to the standard topology.
%\end{lem}
%\begin{proof}
%Since $W$ is finite dimensional, we have a linear isomorphism $\phi:\R^n \to W$, which is continuous. Thus, the topology on $W$ is not stronger than the classical one. Hence it is left to prove that for any $r>0, \,\, \phi(B(0,r))$ contains an open neighborhood of 0, where $B(0,r)$ denotes the open ball with radius $r$ and center at the origin.
%
%Since $W$ is Hausdorff, for any compact $C \subset \R^n, \,\, \phi(C)$ is closed. Consider $\phi(B(0,r) \cup (\R^n \setminus \overline{B(0,2r)}))$. It is an open neighborhood of 0, hence it contains an open convex neighborhood of 0 which in turn must be contained in $\phi(B(0,r))$.
%\end{proof}
%
%
%We immediately obtain
%\begin{cor}\label{cor:FinDense}
%If a locally convex Hausdorff topological  vector space $W$ has a dense finite dimensional subspace then $W$ is finite dimensional.
%\end{cor}

\subsection{Preliminaries on homological algebra} \label{subsec:PrelHomAlg}

%We start with some generalities on abelian categories.
%We will need the following lemmas.
\begin{definition}
Let $\cC$ be an abelian category. A family of objects $\cA
\subset Ob(\cC)$ is called a \textbf{generating} family if for any object $X
\in Ob(\cC)$ there exists an object $Y \in \cA$ and an epimorphism
$Y \twoheadrightarrow X$.
\end{definition}

\begin{definition}[{\cite[III.6.3]{GM}}]
Let $\cC$ and $\cD$ be abelian categories and $\cF: \cC \to \cD$
be a right-exact additive functor. A family of objects $\cA
\subset Ob(\cC)$ is called \textbf{$\cF$-adapted} if it is
generating,
closed under direct sums
and for any acyclic complex  $ \RamiA{\cdots}  \to A_3 \to A_2 \to A_1 \to 0
$ with $A_i \in \cA$, the complex $ \RamiA{\cdots}  \to \cF(A_3) \to \cF(A_2) \to \cF(A_1) \to 0
$ is also acyclic.

%and for any short exact sequence $0 \to A \to B \to C
%\to 0$ with $C \in \cA$, the sequence $0 \to \cF(A) \to \cF(B) \to
%\cF(C) \to 0$ is also exact.

For example,
a generating,
closed under direct sums system consisting of
%if the family of
projective  objects
%is  generating, then it
is $\cF$-adapted for any right-exact %additive
functor $\cF$.
For a Lie algebra, the system of free $\g$-modules (i.e.  direct sums of copies of $U(\g)$) is an example of such system.
\end{definition}
The following results are well-known (see e.g. \cite[III]{GM}).
\begin{theorem}
Let $\cC$ and $\cD$ be abelian categories and $\cF: \cC \to \cD$
be a right-exact additive functor. Suppose that there exists an
$\cF$-adapted family $\cA \subset Ob(\cC)$. Then $\cF$ has derived
functors.
\end{theorem}

\begin{rem}
Note that if the functor $\cF$ has derived functors, and $\cA$ is a generating class of acyclic objects then $\cA$ is $\cF$-adapted.
\end{rem}

\begin{lemma} \label{HomLeib}
Let $\cA$, $\cB$ and $\cC$ be abelian categories. Let $\cF:\cA \to
\cB$ and $\cG:\cB \to \cC$ be right-exact additive functors.
Suppose that both $\cF$ and $\cG$ have derived functors.

(i) Suppose that $\cF$ is exact. Suppose also that there exists a
class $\cE \subset Ob(\cA)$ which is $\cG \circ \cF$-adapted and
such that $\cF(X)$ is $\cG$-acyclic for any $X \in \cE$. Then the functors $L^i(\cG \circ \cF)$ and $L^i\cG \circ \cF$ are isomorphic.

(ii) Suppose that there exists a class $\cE \subset Ob(\cA)$ which
is $\cG \circ \cF$-adapted and $\cF$-adapted and such that
$\cF(X)$ is $\cG$-acyclic for any $X \in \cE$. Let $Y \in \cA$ be
an $\cF$-acyclic object. Then $L^i(\cG \circ \cF)(Y)$ is (naturally) isomorphic to $L^i\cG( \cF(Y))$.

(iii) Suppose that $\cG$ is exact. Suppose that there exists a
class $\cE \subset Ob(\cA)$ which is $\cG \circ \cF$-adapted and
$\cF$-adapted. Then  the functors $L^i(\cG \circ \cF)$ and $\cG \circ L^i\cF$ are isomorphic.
\end{lemma}

%\begin{definition}
%Let $\g$ be a Lie algebra. For any representation $V$ of $\g$ denote $\o\oH_i(G, V):=L^iCI_{\g}(V)$.
%Recall that $CI_G$ denotes the coinvariants functor.
% \end{definition}

\subsubsection{Lie algebra homology} \label{subsubsec:LieAlgCoh}

Let $\g$ be a Lie algebra of dimension $n$ \ and $\pi\in \cM(\g)$. The homology of $\g$ with coefficients in $\pi$ are defined by $\oH_i(\g,\pi):=L^i\cC(\pi)$, where $\cC:\cM(\g) \to Vect$ is the functor of coinvariants and $L^i\cC$ denotes the derived functor. The homology $\oH_i(\g,\pi)$ is equal to the $i$-th homology of the Koszul complex.
%ose homology can be computed using the Koszul complex

\begin{equation} \label{eq:Koszul}
 0 \ot \pi \ot  \g \otimes \pi \ot  \RamiA{\cdots}  \ot\Lambda^n(\g) \otimes \pi \ot 0 .
\end{equation}

Recalling the definition of derived functor one obtains

\begin{lem}\label{lem:LieHom}
Let $0 \to \pi \to \tau \to \sigma \to 0$ be a short exact sequence of $\fg$-modules. Then there exist boundary maps $\partial_i:H_i(\fg,\sigma)\to H_{i-1}(\fg,\pi)$ such that the sequence
\begin{equation*}
0\to H_n(\fg,\pi) \to H_n(\fg,\tau) \to H_n(\fg,\sigma) \overset{\partial_n}{\to} H_{n-1}(\fg,\pi) \to \cdots\to H_1(\fg,\sigma)\overset{\partial_1}{\to} \pi_{\fg} \to \tau_{\fg} \to \sigma_{\fg} \to 0
\end{equation*}
is exact.
\end{lem}

The following lemmas are standard and can be easily deduced from Lemma \ref{HomLeib}
\begin{lem}\label{lem:AcycCompChar}
Let $\fh \subset \g_1 \subset \g$ be Lie algebras.
Let $\psi$  be a character of $\fh$. Suppose that  $\g_1$ stabilizes $(\fh,\psi)$.
 Let $\cF:\cM(\g) \to \cM(\g_1)$ be the functor defined by $\cF(\pi)=\pi_{\fh,\psi}$.  Then
 $L^i(\cF)(\pi)=\oH_{i}(\h,\pi\otimes (- \psi))$.
\end{lem}

\begin{lem}\label{lem:AcycComp}
Let $\g$ be a Lie algebra and $\fri$ be its Lie ideal. Let $\pi$ be a representation of $\g$. Suppose that $\pi$ is $\fri$-acyclic and $\oH_0(\fri,\pi)$ is $\g/\fri$-acyclic. Then $\pi$ is $\g$-acyclic.
\end{lem}

The last lemma is in fact a special case of the Hochschild-Serre spectral sequence.

\subsection{Acyclicity with respect to composition of derivatives (Proof of Lemma \ref{lem:CompD1Acyc})} \label{subsec:PfCompD1Acyc}

By Lemma \ref{HomLeib}(ii), it is enough to show that there exists a class of representations of $\fp_n$ that is $\oPhi$-adapted, $\oE^{k+1}$-adapted, and such that for every object $\tau$ in the class, $\oPhi(\tau)$ is $\oE^{k}$-acyclic.
We claim that the class of free $U(\fp_n)$-modules satisfies this conditions.

For this it is enough to show that $U(\p_n)$ is $\oE^{k+1}$ and $\oPhi$ acyclic and $\oPhi(U(\p_n))$ is  $\oE^{k}$  acyclic. By Lemma \ref{lem:AcycCompChar} it is equivalent to showing that
$$\oH_i(\fu_{n}^{k+1} ,U(\p_n)\otimes (-\psi_{n}^{k+1}))=\oH_i(\fv_{n}, U(\p_n) \otimes (-\psi_n))=\oH_i(\fu_{n-1}^{k}, U(\p_n)_{\fv_n,\psi_n} \otimes (-\psi_{n-1}^{k}))=0,$$
for any $i>0$. This follows from the fact that $U(\p_n)\otimes (-\psi_{n}^{k+1})$ is free $\fu_{n}^{k+1}$ module, $U(\p_n)\otimes (-\psi_{n})$ is free $\fv_n$- module and   $U(\p_n)_{\fv_n,\psi_n} \otimes (- \psi_{n-1}^k)$ is  free $\fu_{n-1}^{k}$-module. This is immediate by the PBW theorem.

\subsection{Preliminaries on topological linear spaces} \label{subsec:PrelTopLin}

We will need some classical facts from the theory of
nuclear \Fre spaces. A good exposition on nuclear \Fre spaces can be found in \cite[Appendix A]{CHM}.

\begin{definition}
We call a complex of topological vector spaces \textbf{admissible}
if all its differentials have closed images.
\end{definition}

\begin{prop}[see e.g. \cite{CHM}, Appendix A] \label{prop:NFS}
$ $
\begin{enumerate}
\item \label{it:FreSES} Let $V$  be a nuclear \Fre space and $W$ be a closed subspace. Then both $W$ and $V/W$ are nuclear \Fre
spaces.

\item \label{it:FreCoh}Let
$$\mathcal{C}: 0 \rightarrow
C_1 \rightarrow  \RamiA{\cdots}  \rightarrow C_n \rightarrow 0$$
be an admissible complex of nuclear \Fre spaces. Then $H^i(\mathcal{C} )$ are   nuclear \Fre spaces.

\item \label{it:FreDual}
Let
$$\mathcal{C}: 0 \rightarrow
C_1 \rightarrow  \RamiA{\cdots}  \rightarrow C_n \rightarrow 0$$
be an admissible complex of nuclear \Fre spaces. Then the complex  $\mathcal{C}^*$ is also admissible and $H_i(\mathcal{C}^* ) \cong
H^i( \mathcal{C}) ^*.$

%\item Let $V$  be a nuclear \Fre space. Then the complex $ \mathcal{C} \ctp V $ is an admissible
%complex of nuclear \Fre spaces and $H^i( \mathcal{C} \hot V)
%\cong H^i( \mathcal{C}) \hot V .$

%\item
%\item $\Sc(\R^n)$ is a nuclear \Fre space.
%
%\item $\Sc(\R^{n+m})= \Sc(\R^n) \ctp \Sc(\R^m)$.
\end{enumerate}
\end{prop}

\begin{cor} \label{cor:CohomDual}
Let $\g$ be a (finite dimensional) Lie algebra and $\pi$ be its nuclear \Fre representation (i.e. a representation in a  nuclear \Fre space such that the action map $\g \otimes \pi \to \pi$ is continuous). Suppose that $\oH_i(\g,\pi)$ are Hausdorff for all $i \geq 0$. Then $\oH_i(\g,\pi)$ are nuclear \Fre and $H^i(\g,\pi^*) = \oH_i(\g,\pi)^*$.
\end{cor}

%\begin{cor} \label{cor:?HomCohom}
%Let $\g$ be a (finite dimensional) Lie algebra and $\pi$ be its nuclear \Fre representation. Suppose that $\pi$ is acyclic and $\oH_0(\g,\pi)$ is Hausdorff. Then
%$\pi^*$ is coacyclic, i.e. $H^i(\g,\pi^*) = 0$ for all $i>0$.
%\end{cor}

%From SmoothTransfer:\\
\begin{prop} [see e.g.  \cite{CHM}, Appendix A] \label{prop:ctp_ext}
\item Let $0 \to V \to W \to U  \to 0$  be an exact sequence of nuclear \Fre spaces. Suppose that the embedding $V \to W$ is closed. Let   $L$  be a nuclear \Fre space. Then the sequence  $0 \to V \ctp L \to W  \ctp L \to U \ctp L \to 0$ is exact and the embedding  $V \ctp L \to W \ctp L$ is closed.
\end{prop}

\subsection{Co-invariants of good extensions (Proof of Lemma \ref{lem:GoodExtExactHaus})} \label{subsec:GoodExtExactHaus}

Let us first prove the following special case.
\begin{lem}\label{lem:BabyGoodExt}
Let $0 \to L\to M \to N \to 0$ be a short exact sequence of nuclear \Fre representations of a Lie algebra $\g$.
Suppose that $L$ and $N$ are acyclic and that $\oH_0(\g,L)$ and $\oH_0(\g,N)$ are Hausdorff.

Then $M$ is acyclic, $\oH_0(\g,M)$ is Hausdorff, and $0 \to \oH_0(\g,L)\to \oH_0(\g,M) \to \oH_0(\g,N) \to 0$ is a short exact sequence.
\end{lem}
\begin{proof}
The long exact sequence implies that $M$ is acyclic and  $0 \to \oH_0(\g,L)\to \oH_0(\g,M) \to \oH_0(\g,N) \to 0$ is a short exact sequence.

By Proposition \ref{prop:NFS}(\ref{it:FreDual}) $0 \to N^* \to M^* \to L^*\to 0$ is exact, by Corollary \ref{cor:CohomDual} $L^*$  is $\g$-coacyclic
(i.e. $\oH^i(\g,L^*)=0$ for $i>0$)
and hence $0 \to (N^*)^{\g} \to (M^*)^{\g} \to (L^*)^{\g} \to 0$ is exact.
We have to show that $\g M$ is closed in $M$. For that it is enough to show that for any $x \in M$, if for any $f \in (M^*)^{\g}$ we have $f(x)=0$ then $x \in \g M$. Let $y$ be the image of $x$ in $N$. We know that any $f \in (N^*)^{\g}$ vanishes on $y$, and hence $y \in \g N$. Hence $\exists x' \in \g M$ such that $x-x'\in L$. Since the map $(M^*)^{\g} \to (L^*)^{\g}$ is onto, every element in $(L^*)^{\g}$ vanishes on $x-x'$. Thus $x-x' \in \g L$ and hence $x\in \g M$.
\end{proof}

We will also need the following lemma.

\begin{lem}\label{lem:ExGr}
Let $C_1 \to  \RamiA{\cdots}  \to C_k$ be a complex of linear spaces. Suppose that each $C_j$ is equipped with a descending filtration $F^i(C_j)$ such that $F^0(C_j)=C_j$, $\bigcap F^i(C_j)=0$ and $\lim_{\leftarrow} C_j/F^i(C_j)\cong C_j$, and $d_j(F^i(C_j)) \subset F^i(C_{j+1})$. Suppose that $F^i(C_1) / F^{i+1}(C_1) \to  \RamiA{\cdots}  \to F^i(C_k)/F^{i+1}(C_k)$ is exact. Then $C_1 \to  \RamiA{\cdots}  \to C_k$ is also exact.
\end{lem}
\begin{proof}
Let $m \in \Ker d_j \subset C_j$. Let $m_j:=[m]\in C_j/F^i(C_j)$.
We want to construct $n \in C_{j-1}$ such that $d_{j-1}(n)=m$.
Since $C_{j-1}\cong \lim_{\leftarrow} C_{j-1}/F^i(C_{j-1}) $ for that it is enough to construct a compatible system of representatives $n_i \in C_{j-1}/F^i(C_{j-1}) $. We build this system by induction. Suppose we constructed $n_i$ and let $n'_i$ be an arbitrary lift of $n_i$ to $C_{j-1}/F^{i+1}(C_{j-1})$. Consider $\eps:=m_{i+1}-d(n'_i)$. Then $\eps \in F^i(C_j)/F^{i+1}(C_j)$. Moreover, $d(\eps)=0$. Hence there exists $ \delta  \in F^i(C_{j-1})/F^{i+1}(C_{j-1})$ such that $d(\delta)=\eps$. Define $n_{i+1}:=n'_i + \delta$.
\end{proof}

\begin{proof}[Proof of Lemma \ref{lem:GoodExtExactHaus}]
Lemma \ref{lem:BabyGoodExt} implies by induction that $\pi / F^i(\pi)$ is acyclic and $A_i:=\oH_0(\g,\pi / F^i(\pi))$ is Hausdorff, and hence is a nuclear \Fre space. Lemma \ref{lem:BabyGoodExt} also implies that $A_i \to A_{i+1}$ is onto. Define $A_{\infty}:= \lim _{\ot} A_i$ (as a linear, non-topological, representation).
Consider the Koszul complex of $\pi$ extended by $ A_{\infty}$:
\begin{equation} \label{eq:KoszulA}
 0 \ot A_{\infty} \ot \pi \ot  \g \otimes \pi \ot  \RamiA{\cdots}  \ot\Lambda^n(\g) \otimes \pi \ot  0 .
 \end{equation}
Lemma \ref{lem:ExGr} implies that the sequence (\ref{eq:KoszulA}) is exact. Hence $\pi$ is acyclic and the natural map $\oH_0(\g,\pi) \to A_{\infty}$  is an isomorphism (of vector spaces).

Let $p_i: \pi \to \pi/F^i(\pi)$ denote the natural projection. The exactness of (\ref{eq:KoszulA}) implies
   $$\g \pi = \bigcap_i p_i^{-1}(\g(\pi/F^i(\pi)))$$
%   \{ v \in \pi \,:\, [v]_i \in \g(\pi/F^i(\pi))\}$$ where $[v]_i$ denotes the class of $v$ in $\pi/F^i(\pi)$. Let
which in turn implies that $\g \pi$ is closed and thus $\oH_0(\g,\pi)$ is a nuclear \Fre space.

Consider the short exact sequence $0 \to F^i(\pi) \to \pi \to \pi/(F^i(\pi))\to 0$.
We showed that it consists of acyclic objects and that $\oH_0(\g,F^i(\pi))$ and $\oH_0(\g,\pi/(F^i(\pi)))$ are Hausdorff. By Lemma \ref{lem:BabyGoodExt} this implies that $0 \to \oH_0(\g,F^i(\pi)) \to \oH_0(\g,\pi) \to \oH_0(\g,\pi/(F^i(\pi))) \to 0$ is a short exact sequence of nuclear \Fre spaces and hence the image of $\oH_0(\g,F^i(\pi)) \hookrightarrow \oH_0(\g,\pi)$ is closed.

Note that $F^i(\oH_0(\g,\pi))=\oH_0(\g,F^i(\pi))$. To sum up, $\oH_0(\g,\pi)$ is a nuclear \Fre space, $ F^i(\oH_0(\g,\pi))$ are closed, the natural morphisms $ F^i(\oH_0(\g,\pi))  \to \oH_0(\g,\pi_i)$ are isomorphisms and $\oH_0(\g,\pi) \cong A_{\infty} \cong \lim \limits _{\ot} \oH_0(\g,\pi)/(F^i(\oH_0(\g,\pi)))$.
%This filtration shows that $\pi$ is a good extension of $\pi_i$.
\end{proof}

\subsection{More on Schwartz functions on Nash manifolds} \label{subsec:PrelScNash}
In this subsection we will recall some properties of  Nash manifolds and
Schwartz functions over them.
We work in the notation
of \cite{AGSc}, where one can read about Nash manifolds and
Schwartz distributions over them. More detailed references on Nash
manifolds are \cite{BCR} and \cite{Shi}.

Nash manifolds are equipped with the \textbf{restricted topology},
in which open sets are open semi-algebraic sets. This is not a
topology in the usual sense of the word as infinite unions of open
sets are not necessarily open sets in the restricted topology.
However, finite unions of open sets are open
 and therefore in the restricted topology we consider only
finite covers. In particular, if $\cE$ over $X$ is a Nash vector bundle
it means that there exists a \underline{finite} open cover $U_i$
of $X$ such that $\cE|_{U_i}$ is trivial.

Fix a Nash manifold $X$ and a Nash bundle $\cE$ over $X$.

An important property of Nash manifolds is
\begin{theorem}[Local triviality of Nash manifolds; \cite{Shi}, Theorem I.5.12  ] \label{loctriv}
Any Nash manifold can be covered by a finite number of open
submanifolds Nash diffeomorphic to $\R^n$.
\end{theorem}

Together with \cite[Corollary 3.6.3]{AGSc} this implies
\begin{theorem} \label{thm:NashSubMan}
Let $Y \subset X$ be a closed Nash submanifold. Then there exist open sunsets $\{U_i\}_{i=1}^n \subset X$ and Nash diffeomorphisms $\phi_i:\R^n \IsoTo U_i$ such that $Y \subset \cup_{i=1}^n U_i$ and $\phi_i^{-1}(U_i \cap Y)$ is a linear subspace of $\R^n$.
\end{theorem}

\begin{theorem}[\cite{AGRhamShap}, Theorem 2.4.16] \label{SurSubSec}
Let  $s:X \rightarrow Y$ be a
surjective submersive Nash map. Then locally it has a Nash
section, i.e. there exists a finite open cover $Y= \bigcup \limits
_{i=1}^k U_i$ such that $s$ has a Nash section on each $U_i$.
\end{theorem}

\begin{corollary} \label{EtLocIsNash}
An \et map $\phi:X \to Y$ of Nash manifolds is locally an
isomorphism. That means that there exists a finite cover $X =
\bigcup U_i$ such that $\phi|_{U_i}$ is an isomorphism onto its
open image.
\end{corollary}

\begin{corollary}[\cite{AGHC}, Theorem B.2.3] \label{NashEquivSub}
Let $p:X \to Y$ be a Nash submersion of Nash manifolds. Then there
exist a finite  open (semi-algebraic) cover $X = \bigcup U_i$ and
isomorphisms $\phi_i:U_i \cong W_i$ and $\psi_i:p(U_i) \cong \cO_i$
where $W_i\subset \R^{d_i}$ and $\cO_i \subset \R^{k_i}$ are open
(semi-algebraic) subsets, $k_i \leq d_i$ and $p|_{U_i}$ correspond
to the standard projections.
\end{corollary}

\begin{notation}
Let $U\subset X$ be an open (Nash) subset. We denote by $\Nash(U,\cE)$ the space of Nash sections of $\cE$ on $U$.
\end{notation}

\begin{defn}\label{def:StrSim}
We call a Nash action of a Nash group $G$ on $X$
\textbf{strictly simple} if it is simple (i.e. all
stabilizers are trivial) and there exists a geometric (separated) Nash quotient $G \backslash X$ (see \cite[Definition 4.0.4]{AGRhamShap}).
\end{defn}

\begin{prop}[\cite{AGRhamShap}, Proposition 4.0.14]\label{prop:GeoQuo}
 Let $H < G < G_n$ be Nash groups. Then the action of $H$ on $G$ is strictly simple.
\end{prop}

\begin{cor}\label{cor:GeoQuo}
 Let $H < G < G_n$ be Nash groups. Let $X$ be a transitive Nash $G$-manifold. Then there exists a geometric quotient $H \backslash X$.
\end{cor}

We will use the following theorem that describes the basic
properties of Schwartz functions on Nash manifolds.

\begin{thm}\label{thm:SchwarProp}
$ $
\begin{enumerate}[(i)]
%\begin{property}[\cite{AGSc}, Theorem 4.1.3]
\item \label{pClass}  $\Sc(\R ^n)$ = Classical
Schwartz functions on $\R ^n$.

%\begin{property}[\cite{AGRhamShap}, Corollary 2.6.2]
\item \label{p:SchFre} The space $\Sc(X,\cE)$ is a nuclear \Fre space.

%\begin{property}[\cite{AGSc}, \S 1.5]
\item \label{prop:ExtClose} Let $Z \subset X$ be a closed Nash submanifold. Then the
restriction maps  $\Sc(X,\cE)$ onto $\Sc(Z,\cE|_Z)$.
%\begin{property}[see \cite{AGSc}, \S 5]
\item \label{pCosheaf}
Let $X = \bigcup U_i$ be a finite open
cover of $X$. Then a global section $f$ of $\cE$ on $X$ is a Schwartz section if
and only if it can be written as $f= \sum _{i=1}^n f_i$
where $f_i \in \Sc(U_i,\cE)$ (extended by zero to $X$).

Moreover, there exists a tempered partition of unity $1 =\sum
 _{i=1}^n \lambda_i$ such that for any Schwartz section $f
\in \Sc(X,\cE)$ the section $\lambda_i f$ is a Schwartz section of $\cE$ on $U_i$ (extended by zero to $X$).

Note that this property implies that Schwartz sections form a cosheaf in the restricted topology.

%\begin{property}[see \cite{AGSc}, \S 5]
\item  \label{pTempSheaf}
Tempered sections of a Nash bundle $\cE$ over a Nash manifold $X$ form a sheaf in the restricted topology. Namely, let $U_i$ be a (finite) open Nash cover of $X$ and let $s \in \Gamma(X,\cE)$. Then $s$ is a tempered section if and only if $s|_{U_i}$ is a tempered section for any $i$.

%\begin{prop}[Schwartz Kernel Theorem, see e.g. \cite{AGRhamShap}, Corollary 2.6.3]
\item \label{SchKer}
For any Nash manifold $Y$,
$$\Sc(X \times Y)=\Sc(X) \ctp \Sc(Y).$$
\end{enumerate}
\end{thm}

%\begin{property}[\cite{AGSc}, Theorem 4.1.3]
Part \eqref{pClass} is \cite[Theorem 4.1.3]{AGSc}, part \eqref{p:SchFre} is \cite[Corollary 2.6.2]{AGRhamShap}, for \eqref{prop:ExtClose} see \cite[\S 1.5]{AGSc}, for
\eqref{pCosheaf} and \eqref{pTempSheaf} see \cite[\S 5]{AGSc}, for \eqref{SchKer} see e.g. \cite[Corollary 2.6.3]{AGRhamShap}.

\begin{lem}[\cite{AOS}, Appendix A]\label{lem:FiltBun}
Suppose
$$
0 \rightarrow \cE_1 \rightarrow \cE_2 \rightarrow \cE_3 \rightarrow 0
$$
is an exact sequence of Nash bundles on $X.$ Then
$$0 \rightarrow \Sc(X,\cE_1) \rightarrow \Sc(X,\cE_2) \rightarrow \Sc(X,\cE_3) \rightarrow 0.
$$
is  an exact sequence of \Fre spaces.
\end{lem}
\Rami{
For the proof we will need the following lemma.
\begin{lemma}\label{lemm:exact_S_loctriv}
Suppose
\[
0\overset{}{\rightarrow} \cE_{1} \overset{d_1}{\rightarrow} \cE_{2} \overset{d_{2}}{\rightarrow} \cE_{3} \overset{}{\rightarrow} 0
\]
is a finite exact sequence of Nash bundles on $X.$ Then there exist an open Nash cover $X= \bigcup_1^n U_j$ s.t. the sequence
$$0\overset{}{\rightarrow} \cE_{1}|_{U_j} \overset{d_1}{\rightarrow} \cE_{2}|_{U_j} \overset{d_{2}}{\rightarrow} \cE_{3}|_{U_j} \overset{}{\rightarrow} 0
$$
is isomorphic to the sequence
\begin{equation}\label{eq:split}
0\overset{}{\rightarrow} {U_j}\times V_1 \overset{}{\rightarrow} {U_j}\times (V_1 \oplus V_2) \overset{}{\rightarrow} {U_j} \times V_2 \overset{}{\rightarrow} 0
\end{equation}
Where $V_i$ are finite dimensional vector spaces, $U_i \times V$ denotes the constant bundle with fiber $V$ and the maps in the sequence come from the the standard embedding and projection. \end{lemma}
This lemma is essentially proven in \cite{BCR}, but for completeness let us prove it here.
\begin{proof}
By the definition of a Nash bundle we can find a finite Nash open cover $X= \bigcup_1^n U_j$ s.t. the bundles $\cE_i|_{U_{i}}$ are constant. Thus we may assume that $\cE_i$ are constant bundles with fibers $W_i$. Choose a basis for $W_2$ and let $I$ be the collection of coordinate subspaces. For any $V \in I$ denote $U_V=\{x \in X | ((d_2)_x)|_V\text{ is an isomorphism} \}$. Clearly $U_V$ form a Nash cover of $X$. Thus we may assume that $X=U_V$ for some $V$. In this case $d_2$ gives an isomorphism between $X \times V$ and $X \times W_3$. Also, $d_1$ gives an isomorphism between  $X \times (W_1 \oplus V)$ and  $X \times W_2$. Those isomorphisms give identification of the sequence $$0\overset{}{\rightarrow} X\times W_1 \overset{}{\rightarrow} X\times W_2 \overset{}{\rightarrow} X \times W_3{\rightarrow} 0$$ with $$0\overset{}{\rightarrow} X\times W_1 \overset{}{\rightarrow} X\times (W_1 \oplus V) \overset{}{\rightarrow} X \times V \overset{}{\rightarrow} 0.$$
We now set $V_1:=W_1, \, V_2:=V$.
\end{proof}
\begin{proof}[Proof of Lemma \ref{lem:FiltBun}]%\ref{lemm:short_exact_S}]
$ $
\begin{enumerate}[Step 1]
\item  The case when the sequence is as in \eqref{eq:split}.\\
It follows immediately from the definition of Schwartz section of a Nash bundle.
\item The general case. \\
%
%We denote by $\pi_i:\cE_i \to X$ the projection of the bundle $\cE_i$ to the %base.
Let $X=\bigcup U_{j}$ be a finite open Nash cover s.t.
the sequence
$$0\overset{}{\rightarrow} \cE_{1}|_{U_j} \overset{d_1}{\rightarrow} \cE_{2}|_{U_j} \overset{d_{2}}{\rightarrow} \cE_{3}|_{U_j} \overset{d_3}{\rightarrow} 0
$$
is isomorphic to a sequence
\begin{equation*}%\label{eq:split2}
0\overset{}{\rightarrow} {U_j}\times V_1 \overset{}{\rightarrow} {U_j}\times (V_1 \oplus V_2) \overset{}{\rightarrow} {U_j} \times V_2 \overset{}{\rightarrow} 0
\end{equation*} as in Lemma \ref{lemm:exact_S_loctriv}. Clearly $d_i\circ d_{i+1}=0$, thus we only have to show that $Ker (d_i)\subset \Im (d_{i-1})$.
By Theorem \ref{thm:SchwarProp}\eqref{pCosheaf} we can choose a partition of unity $1_X= \sum e_j$ where $e_j \in C^{\infty}(X)$ with $\Supp(e_j) \subset U_j$ and such that for any $i,j$ and any $\phi \in \Sc(X,\cE_i)$ we have $e_j \phi \in \Sc(U_j,E_i)$.
Let $\phi \in Ker (d_i) \subset\Sc(X,\cE_{i} )$ then
we have $$\phi=\sum e_j \phi \in \sum \Ker (d_i|_{ \Sc(U_j,\cE_i|_{U_j})})=\sum \Im (d_{i-1}|_{ \Sc(U_j,\cE_{i-1}|_{U_j})}) \subset \Im (d_{i-1}).$$
\end{enumerate}
\end{proof}
}

The following Lemma follows immediately from Corollary \ref{NashEquivSub} and Theorem \ref{thm:SchwarProp}\eqref{pTempSheaf}.
\begin{lem}\label{lem:TempTemp}
Let $p:Y \to X$ be a Nash surjective submersion of Nash manifolds.
Let  $s \in \Gamma(X,\cE)$ be a set-theoretical global section of $\cE$. Then $s$ is tempered if and only if $p^*s$ is a tempered section of $p^*\cE$.
\end{lem}

\begin{lem}[\cite{AGHC},Theorem B.2.4] \label{lem:SubPush}
Let $\phi:Y \to X$ be a Nash submersion of Nash manifolds, and $\cE$ be a Nash bundle over $X$. Then\\
(i) there exists a unique continuous linear map
$\phi_*:\Sc(Y,\phi^!(\cE)) \to \Sc(X,\cE)$ such
that for any $f \in \Sc(X,\cE^*\otimes D_X) $ and $\mu \in \Sc(Y,\phi^!(\cE))$ we have $$\int_{x \in X} \langle f(x),\phi_*\mu(x)
\rangle = \int_{y \in Y} \langle \phi^*f(y), \mu(y) \rangle.$$ In
particular, we mean that both integrals converge. Here, $\phi^!(\cE)=\phi^*(\cE)\otimes D_Y^X$, as in Notation \ref{not:!}.\\
(ii) If $\phi$ is surjective then $\phi_*$ is surjective.
\end{lem}

\begin{lem} \label{lem:NewLeib}[\cite{AGST} ,Lemma B.1.4]
Let $G$ be a connected Nash group and  $\g$ be the Lie algebra of $G$. Let $p:G\times X \to X$ be the projection. Let $G$ act on $\Sc(G \times X,p^!(\cE))$ by acting on the $G$ coordinate. Consider the pushforward $p_*: \Sc(G \times X,p^!(\cE)) \to \Sc(X,\cE)$.
Then
$\g \Sc(G \times X,p^!(\cE)) = \Ker(p_*).$
\end{lem}

\begin{lem}\label{lem:DivSc}
Let $f:X \to \R$ be a Nash function such that $0$ is a regular value of $f$. Then $f\Sc(X) = \{ \phi \in \Sc(X) \, :\, \phi({f^{-1}(0)})=0\}$.
\end{lem}
\begin{proof}
By Theorem \ref{thm:NashSubMan} and partition of unity (Theorem \ref{thm:SchwarProp}\eqref{pCosheaf}), we can assume that  $X = \R^n$ and $f^{-1}(0)= \R^{n-1} \subset \R^n$. Let $x : \R^n \to \R$ be the last coordinate function. Since $0$ is a regular value of $f$, $f/x$ is a smooth invertible function. Since $f$ is a Nash function, $f/x$ is also a Nash function. Thus, we can assume $f=x$.
We have to show that the following sequence is exact.

$$0 \to \Sc(\R^{n}) \overset{x}{\to} \Sc(\R^{n})  \overset{Res}{\to} \Sc(\R^{n-1}) \to 0,$$
where $Res: \Sc(\R^{n}) \to \Sc(\R^{n-1})$ is the restriction map.

By the Schwartz Kernel Theorem (Proposition \ref{SchKer}) this sequence is isomorphic to

$$0 \to  \Sc(\R^{n-1})\ctp \Sc(\R)  \overset{Id \ctp x}{\to} \Sc(\R^{n-1})\ctp \Sc(\R)  \overset{Id \ctp Res}{\to} \Sc(\R^{n-1}) \to 0,$$
where $Res: \Sc(\R) \to \C$ is evaluation at $0$.

By Proposition \ref{prop:ctp_ext} it is enough to prove the exactness in the case $n=1$, which is obvious.
\end{proof}

 We will use the following version of Borel's lemma.
\begin{lem} \label{lem:Borel}
Let $Z \subset X$ be a Nash submanifold.

Then $\Sc_X(Z,\cE)$ has a canonical countable decreasing filtration by closed subspaces $(\Sc_X(Z,\cE))^i$ satisfying
\begin{enumerate}
\item
$\bigcap (\Sc_X(Z,\cE))^i=0$

\item $gr_i(\Sc_X(Z,\cE)) \cong \Sc(Z,\Sym^i(CN_Z^X)\otimes \cE)$, where $CN_Z^X$ denotes the conormal bundle to $Z$ in $X$.

\item the natural map  $$\Sc_X(Z,\cE) \to \lim_\ot(\Sc_X(Z,\cE))/\Sc_X(Z,\cE))^i)$$
is an isomorpihsm.
\end{enumerate}

\end{lem}
For proof see \cite[Lemmas B.0.8 and B.0.9]{AGST}.

\subsection{ Good representations of geometric origin (Proof of Lemma \ref{lem:GenGeoGood})}  \label{subsec:PfGenGeoGood}

For the proof we will fix $T,R \text{ and }R'$. We will need the following statements.

\begin{lem}\label{lem:TranBun}
Let $X$ be an $R$-transitive Nash manifold. Let $\cE$ be an $R$-tempered bundle on $X$. Suppose that for each $x \in X$, the $V_n$-stabilizer of $x$ acts trivially on the fiber $\cE_x$. Note that by Corollary \ref{cor:GeoQuo} there exists  geometric quotient $V_n \backslash X$. Let $p:X \to V_n \backslash X$ denote the projection. Then

\begin{enumerate}

\item  \label{it:star}
there exists an $R'$-tempered bundle $\cE'$ on the geometric quotient $V_n \backslash X$ and a (tempered) isomorphism of $R$-tempered bundles $\cE \IsoTo p^*(\cE')$.

\item \label{it:!} there exists an $R'$-tempered bundle $\cE''$ on the geometric quotient $V_n \backslash  X$ and a (tempered) isomorphism of $R$-tempered bundles $\cE \IsoTo p^!(\cE'')$.
\end{enumerate}
\end{lem}
\begin{proof}

(\ref{it:star})
Let $x \in X$ and let $R_x$ be the stabilizer of $x$.
%Let $Q:=G_{n-1} \times T$
Identify $V_n \backslash X$ with $Y:=R' x$. Let $\cE' := \cE|_{Y}$.

We will construct a (tempered) isomorphism of $R$-tempered bundles $\phi:\cE \IsoTo p^*(\cE')$. In order to do that we have to construct, for any $x \in X$, an element $\phi_x\in \Hom(\cE_x, p^*(\cE')_x)$. Note that $\Hom(\cE_x, p^*(\cE')_x) = \Hom(\cE_x, \cE_{p(x)}).$ We know that there exists $g \in V_n$ such that $gx = p(x)$. This $g$ gives an element in $\phi_{x}\in \Hom(\cE_x, p^*(\cE')_x)$. The element $g$ is not unique, but the assumptions of the lemma imply that $\phi_{x}$ does not depend on the choice of $g$. Now we have to show that the constructed $\phi \in \Gamma(X,\Hom(\cE, p^*(\cE')))$ is a tempered section. Consider the action map $a:V_n \times Y \to X$. Clearly, it is a surjective submersion. It is easy to see that $a^*(\phi) \in \Gamma(V_n \times Y,\Hom(a^*\cE, a^* p^*(\cE')))$ is a tempered section. Thus, Lemma \ref{lem:TempTemp} implies that $\phi \in \Gamma(X,\Hom(\cE, p^*(\cE')))$ is a tempered section.

(\ref{it:!}) By part (\ref{it:star}) there exists a bundle $(D_Y^X)'$ on $V_n \backslash X$ such that $p^*((D_Y^X)') \simeq D_Y^X$. We define $\cE'':=\cE' \otimes ((D_Y^X)')^*$.
\end{proof}

\begin{rem}\label{rem:TranBun}
An analog of this lemma holds for Nash equivariant bundles, and for multiplicative equivariant bundles, with an analogous proof. In particular, the isomorphism $p^*((D_Y^X)') \simeq D_Y^X$ is an isomorphism of Nash bundles.
\end{rem}

\begin{cor}\label{cor:TranBun}
Let $X$ be a transitive $R$-Nash manifold. Let $\cE$ be a $R$-multiplicative bundle on $X$ (see Definition \ref{def:MultBun}). Then there exists a filtration $F^i$ of $\cE$ by $R$-multiplicative bundles such that for every $i$ there exists an $R'$-multiplicative
bundle $\cE_i'$ on the geometric quotient $V_n \backslash X$ such that $F^i(\cE)/F^{i-1}(\cE)=p^!(\cE_i')$.
 \end{cor}
To deduce this corollary we will use the following obvious lemma.
\begin{lem}\label{lem:FinDimNilp} Let $\pi$ be a multiplicative (finite dimensional) representation of a subgroup $L$ of $V_n$. Then the Lie algebra of $L$ acts nilpotently on $\pi$.
\end{lem}

\begin{proof}[Proof of Corollary \ref{cor:TranBun}]
Let $F^i(\cE) \subset \cE$ be defined by $F^i(\cE):=\{(x,e)\in \cE|  ((\fv_n)_x)^{\otimes i} e=0\}$.
By the last lemma (Lemma \ref{lem:FinDimNilp}) the filtration $F^i(\cE)$ is exhaustive. Since $\cE$ is a tempered bundle, the action of the Lie algebra $\fr$ is Nash and hence $F^i(\cE)$ are Nash. Clearly the action $(\fv_n)_x$ on  $F^i(\cE)/F^{i-1}(\cE)$ is trivial. Thus, by Lemma \ref{lem:TranBun}, we are done.
\end{proof}

\begin{proof}[Proof of Lemma \ref{lem:GenGeoGood}]

The case $Y=X$ and $X$ is $R$-transitive follows from Corollary \ref{cor:TranBun} and Lemma \ref{lem:FiltBun}. By the Borel lemma (Lemma \ref{lem:Borel}), this implies the case when $X$ is $R$-transitive and $Y$ is arbitrary. This in turn implies the general case by Corollary \ref{cor:Sc_seq}.
\end{proof}
\subsection{Induced representations as sections of multiplicative bundles
(Proof of Lemma \ref{lem:PrinGood})}\label{subsec:PfPrinGood}

%Let $\pi$ be a principal series representation. By definition, it is induced from a finite dimensional representation of the torus.
%Without loss of generality, by the definition of goodness, we can assume that the inducing representation is a character. We have shown that $\cE$ is a $G$-tempered bundle and thus, by Lemma \ref{lem:GenGeoGood}, $\pi = \Sc(X_n,\cE)$ is good.

%Let $\chi$ be a character of the torus $T_n <B_n<G_n$, continued trivially to $B_n$. Let $\pi = Ind_{B_n}^{G_n}(\chi)$.
Identify the variety $X:=G/H$ with $K/(K \cap H)$. %Clearly, the restriction of $\chi$ to $K_n \cap B_n$ is an algebraic character.
Note that under this identification $\cE\cong K\times_{K \cap H}\chi$. This is a Nash bundle.

Let us show that  the action of $G$ on $\cE$ is multiplicative.
The only non-trivial part is that any $\alpha \in \g_n$ preserves the space of Nash sections of $\cE$ on any open (Nash) subset $U\subset X$.
For this we note that
%the space of Nash sections $\Nash(U,\cE)$ on an open set $U \subset X $  is identified with
$$\Nash(U,\cE) \cong \{f \in C^\infty(p^{-1}(U))|f|_K \text{ is Nash and } f(gh)=\chi(h)f(g) \text{ for any } h \in H\},$$ where $p:G\to X$ is the quotient map.
Under this identification, an element $\alpha \in \g$ acts on $\Nash(U,\cE)$ via a Nash vector field $\beta$ on  $p^{-1}(U).$ This field can be interpreted as a Nash map $\beta:p^{-1}(U)\to \g.$ Here we identify the Lie algebra $\g$ with the tangent space at a point $g \in G$ using the differential of the left translation by $g$. Restrict $\beta$ to $K\cap p^{-1}(U)$ and choose a decomposition $\beta=\beta_1+\beta_2$ with $\beta_1:K\cap p^{-1}(U) \to  \fk$ and $\beta_2:K\cap p^{-1}(U) \to  \fh$, where $\fh$ and $\fk$ denote the Lie algebras of $H$ and $K$ respectively. Let $f \in C^\infty(p^{-1}(U))$ such that $f|_K \text{ is Nash and } f(gh)=\chi(h)f(g)$ for any $h \in H.$ Then $(\beta f)|_K=\beta_1 f|_K+d \chi\circ \beta_2 \cdot f$ which is clearly a Nash function. \proofend

\subsection[Proofs of Lemmas \ref{lem:ProdDer0} and \ref{lem:ProdLastDer}]{Derivatives of quasi-regular representations on $P_n$-orbits on flag varieties (Proofs of Lemmas \ref{lem:ProdDer0} and \ref{lem:ProdLastDer})}\label{subsec:PfLemProd}

\begin{proof}[Proof of Lemma \ref{lem:ProdDer0}]
The proof is by induction on $n$. In step 1 we reduce to the case when $Y=X=P_n / Q_{\lambda}^i$.
Let $Z$ and $Z_0$ to  be as in Lemma \ref{lem:ProdGeo} and $p:X \to Z$ denote the natural projection.
In step 2 we reduce to the case $\cE=p^!(\cE')$, where $\cE'$ is a $G_{n-1}$-multiplicative bundle on $Z$. In step 3 we prove the lemma for this case.

\begin{enumerate}[Step 1]
\item Reduction to the case $Y=X=P_n / Q_{\lambda}^i$.\\
By Lemma \ref{lem:Borel}, $\Sc_Y(X,\cE)$ is a good extension of $\pi_j:=\Sc(X,\cE|_X \otimes Sym^j(CN_X^Y))$. By Lemma \ref{lem:GenGeoGood}, $\pi_j$ are good. Thus, by
Corollary \ref{cor:ExactPf}, $E^{i+1}(\pi)$ is a good extension of $E^{i+1}(\pi _{i})$.
Thus it is enough to show that $E^{i+1}(\pi_j)=0$.

\item Reduction to the case $\cE=p^!(\cE')$, where $\cE'$ is a $G_{n-1}$-multiplicative bundle on $Z$.\\
By Corollary \ref{cor:TranBun}, we have a finite filtration on $\cE$ and $G_{n-1}$-multiplicative bundles $\cE'_i$ on $Z$ such that $p^!(\cE_i)\simeq Gr^i(\cE)$. By Lemma \ref{lem:FiltBun}, this gives a filtration on $\Sc(X,\cE)$ such that $Gr^i(\Sc(X,\cE))= \Sc(X,p^!(\cE_i))$. As before, this means that it is enough to show that $E^{i+1}(\Sc(X,p^!(\cE_i)))=0$.

\item Proof for the case $\cE=p^!(\cE')$.\\
 By the  key lemma (Lemma \ref{lem:GeoGood}), $\DimaA{\oPhi}(\pi)=\Sc(Z_0,\cE')$. Let $\lambda':=\lambda_{k-i+1}^-$, as in Lemma \ref{lem:ProdGeo}. By Corollary \ref{cor:SysRep} and Lemma \ref{lem:ProdGeo}, $Z_0=\bigcup_{j=1}^{i-1} \cO_{\lambda'}^{ij}$. By Corollary \ref{cor:Sc_seq}, there is a filtration $F^j$ on $ \Sc(Z_0,\cE')$ such that $Gr^j(\Sc(Z_0,\cE'))= \Sc_{Z_0}(\cO_{\lambda'}^{ij},\cE')$. Thus it is enough to show that $E^i(\Sc_{Z_0}(\cO_{\lambda'}^{ij},\cE'))=0$ for any $j$. By Corollary \ref{cor:SysRep}, $\cO_{\lambda'}^{ij} \cong P_{n-1}/Q_{\lam '}^j$. Thus, by the induction hypothesis, $E^i(\Sc_{Z_0}(\cO_{\lambda'}^{ij},\cE'))=0$.
\end{enumerate}
\end{proof}

For the proof of Lemma \ref{lem:ProdLastDer}, consider the embedding of $G_{n-1}$ to $G_n$ obtained by conjugating the standard embedding by the permutation matrix corresponding to the permutation $(m_{\lambda}^i, m_{\lambda}^i+1 , \dots ,  n)$. Under this embedding, $P_{\lambda^-_i}$ embeds into $P_{\lambda}$. Denote $\beta_{k}^i:=(|\det|^{-1/2} , \dots ,  |\det|^{-1/2},1,|\det|^{1/2} , \dots ,  |\det|^{1/2}) \in \Chi^k$, where $1$ stands on place number $i$.
We will need the following straightforward computation.
\begin{lem}%\label{lem:ProdLastDer?}
Let $\alp \in \Chi^k$.
 Then $$\chi_{\lam,\alp}|_{P_{\lambda^-_i}} = \chi_{\lam^-_i,\alp \cdot \beta_{k}^i},$$
 where $\alp \cdot \beta_{k}^i \in \Chi^k$ denotes the coordinate-wise product.
\end{lem}

\begin{proof}[Proof of Lemma \ref{lem:ProdLastDer}]
The proof is by induction. Let us first prove the base $i=1$.
Note that $X=P_n/Q_{\lambda}^1\cong G_{n-1}/P_{\lambda_k^-}$. Thus
$\pi|_{G_{n-1}}=\Sc(G_{n-1}/P_{\lambda_k^-},\cE)$. Consider the fiber $\cE_{x_0}$  as a character of $P_{\lambda_k^-}$. We have $\cE_{x_0}=\chi_{\lam,\alp}|_{P_{\lambda_k^-}}=\chi_{\lam^-_k,\alp \cdot \beta_{k}^k}$. Thus $$\pi|_{G_{n-1}}=  \chi_{(n_1),(\alp_1 |\det|^{-1/2})} \times  \RamiA{\cdots}  \times\chi_{(n_{k-1}),(\alp_{k-1}|\det|^{-1/2})} \times \chi_{(n_{k}-1),(\alp_{k})} .$$

For the induction step, we assume $i>0$ and let $Z$, $Z_0$ and $\lam '$ be as in Lemma \ref{lem:ProdGeo} and $p:X \to Z$ denote the natural projection. By Lemma \ref{lem:TranBun} (and Remark \ref{rem:TranBun}) there exists a $G_{n-1}$-multiplicative equivariant bundle $\cE'$ on $Z$ such that $\cE=p^!(\cE')$. By the
 key lemma (Lemma \ref{lem:GeoGood}), $\Phi(\pi)=\Sc(Z_0,\cE')\otimes |\det|^{-1/2}$. By Corollary \ref{cor:SysRep} and Lemma \ref{lem:ProdGeo}, $Z_0=\bigcup_{j=1}^{i-1} \cO_{\lambda'}^{ij}$. By Corollary \ref{cor:Sc_seq}, there is a filtration $F^j$ on  $\Sc(Z_0,\cE')$ such that $Gr^j(\Sc(Z_0,\cE'))= \Sc_{Z_0}(\cO_{\lambda'}^{ij},\cE')$. By Corollary \ref{cor:SysRep}, $\cO_{\lambda'}^{ij} \cong P_{n-1}/Q_{\lam '}^j$. Thus, by Lemma \ref{lem:ProdDer0}, $\Phi^{i-2}(\Sc_{Z_0}(\cO_{\lambda'}^{ij},\cE'))=0$ for $j<i-1$.
Thus
\begin{multline*}
E^i(\pi)=\Phi^{i-1}(\pi) = \Phi^{i-2}(\Phi(\pi)) = \Phi^{i-2}(\Sc_{Z_0}(\cO_{\lambda'}^{i,i-1},\cE')\otimes |\det|^{-1/2})=\\ \Phi^{i-2}(\Sc(\cO_{\lambda'}^{i,i-1},\cE'\otimes |\det|^{-1/2}))=E^{i-1}(\Sc(\cO_{\lambda'}^{i,i-1},\cE'\otimes |\det|^{-1/2})).
\end{multline*}
In the last expression we consider the character $|\det|^{-1/2}$ of $P_{n-1}$ as a constant $P_{n-1}$-equivariant bundle on $Z_0$.

Consider $\nu_1:=\cE'|_{x_{\lambda}^{i,i-1}}$ and $\nu_2:=\cE|_{x_{\lambda}^{i,i-1}}$ as characters of $(P_{n-1})_{x_{\lambda'}^{i,i-1}} = Q_{\lam '}^{i-1}$.
Note that $$\cE'|_{x_{\lambda'}^{i,i-1}} = \cE|_{x_{\lambda'}^{i,i-1}} \otimes D^*_{X}|_{x_{\lambda'}^{i,i-1}} \otimes D_Z|_{x_{\lambda'}^{i,i-1}} $$
and thus $\nu_1 = \nu_2 \otimes \overline{\chi}_{\lam ',\beta}^{i-1}|_{Q_{\lam '}^{i-1}},$
where $\beta=(|\det| , \dots ,  |\det|,1 , \dots ,  1)$, where the last appearance of $|\det|$ is in the place $k-i+1$.
%?? wishful thinking (\beta_k^{k-i+1})^{-1} (|\det|^{1/2} , \dots ,  |\det|^{1/2})=
%(|\det| , \dots ,  |\det|,|\det|^{1/2},1 , \dots ,  1)$
Thus
\begin{multline*}(\cE'\otimes |\det|^{-1/2})|_{x_{\lambda}^{i,i-1}}= \nu_1 \otimes |\det|^{-1/2} = \nu_2  \cdot  \overline{\chi}_{\lam ',\beta}^{i-1}|_{Q_{\lam '}^{i-1}} \cdot  |\det|^{-1/2}  =\\
(\chi_{\lam,\alp}^i|_{Q_{\lam '}^{i-1}} \cdot  \overline{\chi}_{\lam ',\beta}^{i-1}|_{Q_{\lam '}^{i-1}} \cdot  |\det|^{-1/2})| =
((\chi_{\lam,\alp}^i|_{P_{\lam'}^{i-1}})\cdot  \overline{\chi}_{\lam ',\beta}^{i-1} \cdot  |\det|^{-1/2})|_{Q_{\lam '}^{i-1}}= \\
(\chi_{\lam ',\alp \cdot \beta_{k}^{k-i+1}}^{i-1}  \cdot  \overline{\chi}_{\lam ',\beta}^{i-1} \cdot  |\det|^{-1/2})|_{Q_{\lam '}^{i-1}} = \chi_{\lam ',\alp'}^{i-1}|_{Q_{\lam '}^{i-1}},
\end{multline*}
where $\alp'=(\alp_1 , \dots ,  \alp_{k-i},\alp_{k-i+1}|\det|^{1/2},\alp_{k-i+2} , \dots ,  \alp_k)$.

Thus, by the induction hypothesis,
\begin{multline*}E^{i}(\pi)|_{G_{n-i}} = E^{i-1}(\Sc(\cO_{\lambda}^{i-1,j},\cE'\otimes |\det|^{-1/2})) = \\
\chi_{(n_1),(\alp_1 |\det|^{-1/2})} \times  \RamiA{\cdots}  \times\chi_{(n_{k-i}),(\alp_{k-i}|\det|^{-1/2})} \times \chi_{(n_{k-i+1}-1),(\alp_{k-i+1})} \times  \RamiA{\cdots}  \times \chi_{(n_k-1),(\alp_k)}
\end{multline*}
\end{proof}

\section{Homology of geometric representations and the proof of the  key lemma (Lemma \ref{lem:GeoGood})} \label{sec:PfGeoGood}
\subsection{
Sketch of the proof}
We have to compute $\fv_n$-homology of the representation
$\Sc(X,\cE)\otimes (-\psi)$.
We first note that this task is local on $X'=V_n\backslash X$, i.e. if we cover $X'$ by subsets $U_i$ it is enough to compute the homology of $\Sc(p^{-1}(U_i),\cE)\otimes (-\psi)$. Now we can cover $X'$ by refined enough cover s.t. for each $U_i$ , the space $p^{-1}(U_i)$ will look like a product $U_i \times W$ where $W$ is  a $V_n$-orbit and the bundle $\cE|_{U_i}$ is trivial. Thus we reduced to computation of homology of  $\Sc(U_i \times W)\otimes (-\psi)$.

Note that the action of $V_n$ on $U_i \times W$ is not the usual product action but rather a twisted product, i.e. the action $V_n$ on its orbit $\{x\} \times W$ depends on the point $x \in U_i$. We can untwist this product (see Proposition \ref{prop:Actions}) but this will cause a twist in the character $\psi$. Namely, it will replace it by a line bundle $\cE$ where the action of $V_n$ on $\cE_{\{x\} \times W}$ depends on the point $x \in U_i$. Thus we reduce to computation of homology of $\Sc(U_i \times W,\cE)$.

Now we use a relative version of the Shapiro Lemma (see Theorem \ref{ShapLem}) in order to reduce to the computation of the homology of $\Sc(U_i,\cE)$ as a representation of the stabilizer $(V_n)_0$ in $V_n$ of a point  $0 \in W$.  Note that the action of $(V_n)_0$ on $U_i$ is trivial and thus we can view $\Sc(U_i,\cE) $  as a family of characters, i.e. it is defined by a map $\phi$ from $U_i$  to the space $(V_n)_0^*$ of characters of $(V_n)_0$. In Lemma \ref{lem:SimpHo} we compute homology of such families under the assumption that $\phi$ is submersive at the trivial character. This assumption is satisfied in our case due to the action of $P_n.$

\begin{remark}
The proof is based on series of reductions. If we were interested only in acyclicity then one could give a relatively simple proof in which each of those reductions is given by general statements. However we are also interested in computation of $\oH_0$ and this makes those reductions more complicated. For example in the first step when we say that the computation is local one should explain what does this mean. We do it by constructing an explicit morphism $I:\oH_0(\fv_n,\Sc(X,\cE)\otimes (-\psi)) \to \Sc(X_0',\cE|_{X_0'})$ and proving that this morphism is an isomorphism rather than proving that there exists some isomorphism. This forces us to make each reduction more explicit. This is sometimes unpleasant since some of the reductions e.g. the Shapiro lemma are using the general machinery of homological algebra which usually is not so explicit.

Therefore  if the reader is not interested in all the details we recommend him to skip all the parts that regard $\oH_0$ and concentrate only on the acyclicity. The computation of $\oH_0$ is essentially the same but its exposition is much longer.
\end{remark}

%==========================================\\
%We will need the following lemmas.
\subsection{Ingredients of the proof}
We will need the following version of Shapiro lemma.

\begin{theorem}[Relative Shapiro lemma]\label{ShapLem}
Let $G$ be an affine Nash group and $X$ be a transitive Nash $G$-manifold.
Let $Y$ be a Nash manifold.
Let $x \in X$ and denote $H:=G_x$. Let $\cE\rightarrow X \times Y$ be a
$G$ equivariant Nash bundle.
Suppose that $G$ and $H$ are homologically trivial (i.e. all their homology except $\oH_0$
vanish and $\oH_0=\R$).
Let $\Chi:G \times Y \to \R$ be a Nash map such that for any $y\in Y$, the map $\Chi|_{G \times \{y\}}$ is a group homomorphism.
Let $\Chi'$ be  1-dimensional $G$-equivariant bundle on $Y$, with action of $G$ given by $g(y,v)=(y,\theta(\Chi(g,y))v)$.
Let $\cE' := \cE \otimes (\C \boxtimes \Chi')$, where $\C$ denotes the trivial bundle on $X$ and $\boxtimes$ denotes exterior tensor product.
%
%Let $\cE'$ be a $G$-equivariant bundle on $ M \times N$ defined in the following way: as a vector bundle $\cE'$ is the same as $\cE$, and the $G$ action is %given by $g(m,n,v)$
Then
$$\oH_i(\mathfrak{g},\Sc(X\times Y,\cE'))\cong \oH_i(\mathfrak{h},\Sc(\{x\}\times Y,\cE'_{\{x\}\times Y}\otimes \Delta_{H} \cdot (\Delta_G^{-1})|_H)),$$
where $\Delta_{H}$ and $\Delta_G$ denote the modular characters of the groups $H$ and $G$.
\end{theorem}
The proof of this theorem is along the lines of the proof of \cite[Theorem 4.0.9]{AGRhamShap}.
We give it in Appendix \ref{app:PfShapLem}.

%\begin{lem}[Relative Shapiro Lemma]\label{lem:FrobShap}
%Let $G$ be a Nash group.
%Let $X$ and $Y$ be $G$-Nash manifolds. Suppose that $Y$ is $G$-transitive and  that $G$ and $G_y$ are unimodular. Suppose that $Y$, $G$ and $G_y$ are contractible. Let $\cE$ be a $G$-equivariant tempered bundle on $X$. Then $\oH_*(\g, \Sc(X,\cE))=\oH_*(\g_y, \Sc(\cE_{X_y})).$
%\end{lem}
%

\begin{lem}\label{lem:SimpHo}
Let $X$ be a Nash manifold and $V$ be a real vector space. Let $\phi:X \to V^*$ be a Nash map. Suppose that $0 \in V^*$ is a regular value of $\phi$. It gives a map $\chi: V \to \cT(X)$ given by $\chi(v)(x)= \theta(\phi(x)(v))$ (recall that $\cT(X)$ denotes the space of tempered functions). This gives an action of $V$ on $S(X)$ by $\pi(v)(f):= \chi(v) \cdot f$. Then
\begin{enumerate}[(i)]
\item \label{it:SimpAcyc}
$\oH_i(\fv,\Sc(X))=0$ for $i>0$.

\item \label{it:SimpH0} Let $X_0:=\phi^{-1}(0)$. Note that it is smooth. Let $r$ denote the restriction map $r:\Sc(X) \to \Sc(X_0)$. Then $r$ gives an isomorphism $\oH_0(\fv,\Sc(X)) \IsoTo \Sc(X_0)$.
\end{enumerate}
\end{lem}
We will prove this lemma in \S \ref{subsec:PfSimpHo}

\begin{notation}\label{not:tau}
In the situation of the Lemma we will denote the action of $\fv$ on $\Sc(X)$ by $\tau_{X,\phi}$.
\end{notation}

In order to check the conditions of Lemma \ref{lem:SimpHo} we will need the following lemma that we will prove in \S \ref{subsec:Pf0Reg}.

\begin{lem}\label{lem:0Reg}
Let a Nash group $G$ act linearly on a finite-dimensional vector space $V$ over $F$, such that the action on $V^* \setminus 0$ is transitive. Let $Q$ be a closed Nash subgroup of $G$ and $L$ be a vector subspace of $V$ stabilized by $Q$. Let $\varphi \in V^*$ be a functional. Consider the map $a:G \to L^*$ defined by $a(g):=g\varphi|_L$. Let $U \subset G/Q$ be an open (Nash) subset and $s : U \to G$ be a local Nash section of the canonical projection $p:G \to G/Q$. Then $0$ is a regular value of $\mu := a \circ s$.
\end{lem}

\begin{proposition}\label{prop:Actions}
Let a Nash group $G$ act transitively on a Nash manifold  $Y$ and let $X$ be a Nash manifold.

Let $X$ ``act" on $G$ i.e. let $G'$ be a Nash group acting on $G$ by automorphisms and $a:X \to G'$ be a Nash map.
This defines a twisted action of $G$ on $X \times Y$. More precisely $\rho_1(g)(x,y)=(x,a(x)(g)(y))$. Let $\chi$ be a fixed tempered character of $G$.
Let $\rho_2$ denote the non-twisted action of $G$ on $X \times Y$, i.e. $\rho_2(g)(x,y)= (x,gy)$.

Define function $\Chi(g,x,y) = \chi((a(x))^{-1}(g))$. Note that it does not depend on $y$. It defines a tempered $(G,\rho_2)-$equvariant structure on the trivial line bundle on $X \times Y$. We denote the resulting bundle by $\cE$.
Let $\pi_1$ denote the representation of $\g$ on $\Sc(X \times Y)\otimes \chi$ given by the action $\rho_1$ and $\pi_2$ denote the representation of $\g$ on $\Sc(X \times Y,\cE)$ given by the action $\rho_2$.

Then $\oH_*(\g,\pi_1)= \oH_*(\g,\pi_2)$.
\end{proposition}
We will prove this proposition in \S \ref{subsec:PfActions}.

\begin{notation}\label{not:pi}
In the situation of the proposition we will denote $\rho_1$ by $\rho_{a,X,Y}$ and $\pi_1$ by $\pi_{a,X,Y}$.
\end{notation}

\subsection{Proof of the  key lemma}
First, let us prove the following version of the lemma.
\begin{lem}\label{lem:TwistProd}
Let $X'$ be a Nash manifold. Let $0 \to L \to V \overset{p}{\to} W \to 0$ be an exact sequence of finite dimensional vector spaces over $F$. Let $\nu: X'\to V^*$ be a Nash map, such that $0$ is a regular value of the composition $X'\to V^* \to L^*$. Let $X'_0$ be the preimage of $0$ under this composition. Note that it is smooth.
Fix a Haar measure on $W$.
%Let $V$ act on $O \times W$ through the second factor.
Let $\pi_{\nu}$ be the representation of $V$ on $\Sc(W \times X')$ given by $$\pi_{\nu}(v)(f)(z,w):= \theta(\langle \nu(z),v\rangle)f(z,w+p(v)).$$ Then
\begin{enumerate}[(i)]
\item \label{it:Acyc} $\pi_{\nu}$ is acyclic
\item \label{it:H0} Note that $\nu(X'_0) \subset W^*\subset V^*$. Let $\nu_0:X'_0 \to W^*$ be the map given by restriction of $\nu$. Let $F \in \T(X'_0 \times W)$ be given by $F(z,w):=\langle \nu_0(z) , w\rangle$. Consider the map $I_{\nu}:\Sc(W \times X') \to \Sc(X'_0)$ given by $I_{\nu}(f):=(pr_{X'})_*(f|_{X'_0 \times W} \cdot F)$, where $pr_{X'}: W \times X' \to W$ denotes the projection.
    Then $I_{\nu}$ defines an isomorphism $\oH_0(\fv,\pi_{\nu}) \IsoTo \Sc(X'_0)$.
\end{enumerate}
\end{lem}
\begin{proof}
(\ref{it:Acyc}) By the relative Shapiro Lemma (Lemma \ref{ShapLem}) it is enough to show that $\tau_{X',\nu}$ (see Notation \ref{not:tau}) is acyclic. This follows from Lemma \ref{lem:SimpHo}.

(\ref{it:H0}) Fix a Haar measure on $L$. Since we fixed a Haar measure on $W$, this defines a Haar measure on $V$ as well. Define $\bar{F} \in \T(X' \times V)$ by $\bar{F}(z,v):=\langle \nu(z) , v\rangle$.
Define $\bar{I}:\Sc(X' \times V) \to \Sc(X')$ by $\bar{I}(f)=(pr_{X'})_*(f) \cdot \bar{F}.$ Let $r:\Sc(X') \to \Sc(X'_0)$ denote the restriction. Let $\mu : X' \to L^*$ be the composition of $\nu$ with $V^* \onto L^*$. Define an action $\tilde{\pi}_{\nu}$ of $V$ on $\Sc(X' \times V)$ given by $$\pi_{\nu}(v)(f)(z,v'):= \theta(\langle \nu(z),v\rangle)f(z,v'+v).$$
Define an action $\sigma$ of $L$ on $\Sc(X' \times V)$ by $\sigma(l)(f)(z,v'):= f(z,v'+l).$

Note that the following diagram is commutative

 \xymatrix{ & \parbox{60pt}{\center{(1)}} & \parbox{100pt}{\center{(2)}} & \parbox{100pt}{\center{(3)}} & \\
(1)&   & \parbox{60pt}{$\fv\otimes \Sc(X' \times V)$}\ar@{->}[r]^{Id_{\fv}\otimes(Id_{X'} \times p)_*}\ar@{->}^{d \tilde{\pi}_\nu}[d] & \parbox{60pt}{$\fv \otimes \Sc(W \times X')$\ar@{->}^{d\pi_\nu}[d]} & \\
(2)& \parbox{60pt}{$\fl \otimes \Sc(X' \times V)$}\ar@{->}^{d\sigma}[r]\ar@{->}^{Id_\fl\otimes\bar{I}}[d]&\parbox{45pt}{$\Sc(X' \times V)$}\ar@{->}^{(Id_{X'} \times p)_*}[r]\ar@{->}^{\bar{I}}[d]
& \parbox{50pt}{$\Sc(W \times X')$} \ar@{->}[r]\ar@{->}^{I}[d]
& 0 \\
(3)& \parbox{40pt}{$\fl \otimes \Sc(X')$}\ar@{->}^{\tau_{X',\mu}}[r]\ar@{->}[d]& \parbox{25pt}{$\Sc(X')$}\ar@{->}^{r}[r]\ar@{->}[d] & \parbox{30pt}{$\Sc(X'_0)$} \ar@{->}[r]\ar@{->}[d] & 0 \\
 &                   0                                       &           0                                   &                   0                          &
}
Recall that $\tau_{X',\mu}$, used in the diagram, is defined in Notation \ref{not:tau}.
%Note that the map $\bar{I}$ intertwines the action $\sigma$ and $\tau_\mu$ (see Notation \ref{not:tau}), the map $(Id \times p)_*$ intertwines $\tilde{pi}_\nu$ to $\pi_\nu$.
It is enough to show that column (3) is exact.
First, let us show that column (2) is exact. The map $\bar{I}$ is onto by Lemma \ref{lem:SubPush} and the exactness in the place of $\Sc(X' \times V)$ follows from \ref{lem:NewLeib}. The exactness of row (2) is proven in the same way. The exactness of row (3) follows from Lemma \ref{lem:SimpHo}. The exactness of column (2) implies the exactness or column (1). Let us prove that column (3) is exact. First, note that $I$ is onto by Lemma \ref{lem:SubPush} and Theorem \ref{thm:SchwarProp}\eqref{prop:ExtClose}. Since $I \circ d\pi_{\nu}=0$, we have $\Im(d\pi_{\nu})\subset \Ker I$, and it is left to show the other inclusion.
%\ref{lem:ExtClosed}.

Let $f \in \Ker(I) \subset \Sc(W \times X').$ Let $\tilde{f} \in \Sc(X' \times V)$ be its preimage under $(Id \times p)_*$. Then $\bar{I}(\tilde{f}) \in \Ker(r) = \Im (\tau_{\mu})$. Let $h\in \fl \otimes \Sc(X')$ be a preimage of $\bar{I}(\tilde{f})$ and let $\tilde{h}$ be a preimage of $h$ in $\fl \otimes \Sc(X' \times V)$. Then $d\sigma(\tilde{h}) - \tilde{f} \in \Ker(\bar{I}) = \Im d \tilde{\pi}_{\nu}$. Now,
$$ f = (Id \times p)_*(\tilde{f}) = (Id \times p)_*(\tilde{f} -d\sigma(\tilde{h}) +  d\sigma(\tilde{h})) = (Id \times p)_*(\tilde{f} -d\sigma(\tilde{h})) \in (Id \times p)_*(\Im d \tilde{\pi}_{\nu}) \subset \Im(d\pi_{\nu}) $$
\end{proof}

We will prove the following Lemma which clearly implies the  key lemma.
\begin{lem}\label{lem:VeryGenGeoGood}
Let $T$ be a Nash linear group and let $R:=P_n \times T$ and $R':=G_{n-1}\times T$.
Let $Q < R$ be a Nash subgroup and let $X \subset R/Q$ be a $V_n$-invariant open Nash subset and $X':=V_n\backslash X$.
Note that $X'$ is an  open Nash subset of $R'/Q',$ where $Q'=Q/(Q \cap V_n)$.
Let $X_0= \{ x \in X \, : \, \psi|_{(V_n)_x}=1.\}$ Let $X_0'$ be the image of $X_0$ in $X'$.  Let $\cE'$ be a $R'$-equivariant tempered bundle on $R'/Q'$ and $\cE:=p_{X'}^!(\cE'|_{X'})$.
Then
\begin{enumerate}[(i)]
\item $\oH_i(\fv_n,\Sc(X,\cE)\otimes (-\psi))=0$ for any $i>0$.
\item $X_0'$ is smooth
\item
Since $X' \subset R'/Q'$ and $X\subset R/Q$ we have a canonical section to $p_{Z}$ and will consider $X'$ as a closed submanifold of $X$.
Consider the action map $a: X_0' \times V_n \to X_0$. Let $\DimaA{\Xi} = 1 \boxtimes \psi \in \T(X_0' \times V_n)$. Note that $\DimaA{\Xi}_0$ is constant along the fibers of $a$. Define $\DimaA{\Xi} \in \T(X_0)$ such that $a^*(\DimaA{\Xi})=\DimaA{\Xi}_0$. Consider the map $I:\Sc(X, p^{!}\cE) \to \Sc(X_0',\cE|_{X_0'})$ given by $I(f):=(p_{X'}|_{X_0})_*(f|_{X_0'} \cdot \DimaA{\Xi})$.
Then $I$ gives an isomorphism $\oH_0(\fv_n,S(X,p^{!}\cE)) \IsoTo \Sc(X_0',\cE|_{X_0'})$,
 as representations of $\fp_{n-1}$.
\end{enumerate}
\end{lem}

\begin{rem}
The open subset $X \subset R/Q$ is not necessary $P_{n-1}$-invariant, but $\fp_{n-1}$ will still act on the homology.
\end{rem}

First let us prove the following special case.
\begin{lem}\label{lem:SpecGenGeoGood}
Lemma \ref{lem:VeryGenGeoGood} holds under the assumption that there exists a Nash section $a:X' \to R'$ of the quotient map $R' \to R'/Q'$.
\end{lem}

\begin{proof}
Consider $R'/Q'$ as a subset of $R/Q$.
Without loss of generality we assume that the class of the unity element lies in $X'$. We will denote it by $z_0$. Without loss of generality we assume $a(z_0)=1$.
Let $W:=p_{X'}^{-1}(z_0)$. Let $L:=(V_n)_{i(z_0)}=Q \cap V_n$. It is a Nash subgroup of $V_n$ and hence is a real vector space.
Note that $W \cong V_n/L$.
%Define $a:{X'} \to G_{n-1}=GL(V)$ by $a(z):=s(z)$.
Define $\phi:{X'} \times V_n \to X$ by
$\phi(z,v):= (a(z)v)z$, where $a(z)v$ denotes the action of $a(z)\in G_{n-1}$ on $v\in V_n$. Note that $\phi$ factors through ${X'} \times W$ and let $\overline{\phi}$ be the corresponding map ${X'} \times W \to X$.
It is easy to see
%We will show
that $\overline{\phi}$ is a Nash diffeomorphism and intertwines $\rho_{a,{X'},W}$ with the action of $V_n$ on $X$ (see Notation \ref{not:pi}).

Note that the section $a$ gives a tempered isomorphism between $\cE'|_{X'}$ and the trivial bundle with fiber $\cE'_{z_0}$. Thus $\overline{\phi}$ gives an $V_n$ -isomorphism $\overline{\phi}^*:\Sc(X,\cE)\otimes (-\psi) \IsoTo \pi_{a,{X'},W} \otimes \cE'_{z_0}.$
Define $\nu:{X'} \to V_n^*$ by $\nu(z)(v):=\bar{\psi}(a(z)v)$.
Note that $X_0'$ coincides with $X'_0$ from Lemma \ref{lem:TwistProd}.
Note that the underlying vector space of $\pi_{a,X',W}$ coincides with the underlying vector space of $\pi_{\nu}$ from Lemma \ref{lem:TwistProd}. Thus we can consider $I_{\nu}\otimes Id$ as a map
$\pi_{a,X',W} \otimes \cE'_{z_0} \to \Sc(X_0')\otimes \cE'_{z_0} \cong \Sc(X_0',\cE'|_{X_0'})$. Note that $\overline{\phi}^*$ intertwines the map $I$ to $I_{\nu}\otimes Id$.

By Proposition \ref{prop:Actions}, $\oH_*(\fv_n,\pi_{a,X',W}) \cong \oH_*(\fv_n,\pi_\nu)$ and on $\oH_0$ the isomorphism is just equality of quotient spaces. Thus it is enough to show that $\pi_\nu$ is acyclic and $I_\nu$ defines an isomorphism $\oH_0(\fv_n,\pi_\nu) \IsoTo \Sc(X'_0)$. To show this, by Lemma \ref{lem:TwistProd} it is enough to show that $0$ is a regular value of the composition $\nu_L: X'\overset{\nu}{\to} V_n^* \to L^*$.
This follows from Lemma \ref{lem:0Reg}.

\end{proof}

\begin{proof}[Proof of Lemma \ref{lem:VeryGenGeoGood}]

%Let $r:\Sc(X,\cE) \to \Sc(X_0,\cE|_{X_0})$ denote the restriction.
%Since $X'=G_{n-1}Q'$ and $X=P/Q$ we have a canonical section to $p_X'$ and will consider $X'$ as a closed submanifold of $X$.
%Consider the action map $a: X_0' \times V \to X_0$. Let $\DimaA{\Xi} = 1 \boxtimes \psi \in \T(X_0' \times V)$. Note that $\DimaA{\Xi}_0$ is constant along the fibers of $a$. Define $\DimaA{\Xi} \in \T(X_0)$ such that $a^*(\DimaA{\Xi})=\DimaA{\Xi}_0$. Consider the map $I:\Sc(X, p^{!}\cE) \to \Sc(X_0',\cE|_{X_0'})$ given by $I(f):=(p_{X'}|_{X_0})_*(f|_{X_0'} \cdot \DimaA{\Xi})$. Note that $I$ is a morphism of $P_{n-1}$-representations.
%We will show that $S(X,p^{?}\cE)$ is acyclic and that $I$ gives an isomorphism %$\oH_0(S(X,p^{!}\cE)) \IsoTo \Sc(X_0',\cE|_{X_0'})$.

Let $q:R' \to R'/Q'$ be the quotient map. By Theorem \ref{SurSubSec}, there exists a finite Nash cover $\{U_i\}$ of $X' \subset R'/Q'$ and sections $s_i:U_i \to R'$ of $q$. By Lemma \ref{lem:SpecGenGeoGood}, $\Sc(p_{X'}^{-1}(U_i),\cE|_{p_{X'}^{-1}(U_i)})$ is acyclic and $I|_{\Sc(p_{X'}^{-1}(U_i),\cE|_{p_{X'}^{-1}(U_i)})}$ gives an isomorphism $$\oH_0(\fv_n,\Sc(p_{X'}^{-1}(U_i),\cE|_{p_{X'}^{-1}(U_i)})) \IsoTo \Sc(U_i \cap X_0',\cE'|_{U_i \cap X_0'}).$$
Consider the extended Kozsul complex of $\Sc(X,\cE)\otimes (-\psi)$:
\begin{equation}
0\to \Lambda^n(\fv_n)\otimes \Sc(X) \otimes (-\psi) \overset{d_{n}}{\to}  \RamiA{\cdots}  \overset{d_{2}}{\to} \fv_n \otimes \Sc(X) \otimes (-\psi) \overset{d_{1}}{\to} \Sc(X) \otimes (-\psi) \overset{d_{0}:=I}{\to} \Sc(X_0',\cE'|_{X_0'}) \to 0
\end{equation}

We have to show that it is exact.
The fact that $I$ is onto follows from Theorem \ref{thm:SchwarProp}\eqref{prop:ExtClose} and Lemma \ref{lem:SubPush}.
Let us show that it is exact in place $l \geq 0$, i.e. at the object $\Lambda^l(\fv_n)\otimes \Sc(X)$.
%using the fact that $\oH_l(\Sc(U_i)\otimes (-\psi))=0$.
Choose a partition of unity $e_i$ corresponding to the cover $X' = \bigcup U_i$.
Clearly $\Im d_{l+1}\subset \Ker d_l$. To show the other inclusion, let $\alpha \in \Ker d_l \subset \Lambda^l(\fv_n) \otimes \Sc(X) \otimes (-\psi)$. Consider $$p^*(e_i)\alpha \in\Ker d_l|_{\Lambda^l(\fv_n) \otimes \Sc(p^{-1}(U_i) \otimes (-\psi))}.$$
%
%Since  $\oH_l(\Sc(U_i)\otimes \psi)=0$
By Lemma \ref{lem:SpecGenGeoGood},
 we have $$p^*(e_i)\alpha \in \Im(d_{l+1}|_{\Lambda^l(\fv_n) \otimes \Sc(p^{-1}(U_i) \otimes (-\psi))}) \text{ and thus }\alpha = \sum p^*(e_i)\alpha \in \Im(d_{l+1}).$$
 \end{proof}

\subsection{Untwisting a product (Proof of Proposition \ref{prop:Actions})}\label{subsec:PfActions}

%{prop:Actions}
\begin{lemma}
Let $\g$ be a Lie algebra and $A$ be a commutative algebra with $1$. Let $M$ be a $(U(\g)\otimes A)$-module.
Then $\oH_*(\g,M)= \oH_*(\g \otimes A,M)$.
\end{lemma}
\begin{proof}
Let $\cA$ be the category of $(U(\g)\otimes A)$-modules and $\cB$  be the category of $\g$-modules. Let $\cF:\cA \to \cB$ be the forgetful functor and $\cG:\cB \to Vect$ be the functor of $\g$ co-invariants. Note that $\cG \circ \cF$ is the functor of coinvariants with respect to the Lie algebra $\g \otimes A$. By Lemma \ref{HomLeib} we have $$L^i(\cG \circ \cF)=L^i(\cG) \circ \cF.$$
This proves the assertion.
\end{proof}

\begin{proof} [Proof of Proposition \ref{prop:Actions}]
%1.We replace $G$ with $\g$ (probably in the formulation too).
Extend the representations $\pi_i$ to representations $\Pi_i$ of $\T(X) \otimes \g \cong \T(X,\g).$

Let $b:  \T(X,\g) \to  \T(X,\g)$ be given by $b(f)(x)=a(x)(f(x)).$ Note that $b$ is invertible.
Let us show that $\Pi_1 \cong \Pi_2 \circ b.$

First, note that as linear spaces both can be identified with $\Sc(X\times Y)$.
Now, denote by $d\rho_i$ the corresponding maps from $\g$ to the space of vector fields on $X \times Y$ and by $d\Chi : \g \times X \times Y$  the differential of $\Chi$ in the first variable. Let $f \in \T(X,\g), \, h \in  \Sc(X\times Y), \text{ and }(x,y)\in X\times Y$. Then
\begin{multline*}
(\Pi_1(f)h)(x,y) = (d\rho_1(f(x))h)(x,y) + \chi(f(x))h(x,y)= (d\rho_2(a(x)(f(x)))h)(x,y)+\chi(f(x))h(x,y)= \\
(d\rho_2(a(x)(f(x)))h)(x,y)+\chi(a(x)^{-1}(a(x)(f(x))))\cdot h(x,y)= \\
(d\rho_2(a(x)(f(x)))h)(x,y)+d\Chi(a(x)(f(x)),x,y) \cdot h(x,y)= \\
(d\rho_2((b(f)(x))h)(x,y)+d\Chi(b(f)(x),x,y) \cdot h(x,y)= (\Pi_2(b(f))(h))(x,y)
\end{multline*}
Now
$$\oH_*(\g,\pi_2)= \oH_*(\T(X) \otimes \g,\Pi_2) \cong \oH_*(\T(X) \otimes \g,\Pi_1)= \oH_*(\g,\pi_1).$$

%2. We replace $\g$ with $\g \otimes \Sc(X)$.
%3. We make an automorphism of  $\g \otimes \Sc(X)$ which untwists the action.
%4.  We replace back $\g \otimes \Sc(X)$ with $\g$, but now the action is not twisted.
%5. We deduce the assertion from the usual Shapiro lemma using properties of completed tensor product of Nuclear Frechet spaces as in the appendix of \cite{CHM}.
\end{proof}

\subsection{Homology of families of characters (Proof of Lemma \ref{lem:SimpHo})} \label{subsec:PfSimpHo}

We prove the lemma by induction on $\dim V$.

Base: $\dim V =1$. In this case we have to show that the following extended Koszul complex is exact:

%$$0 \to V \otimes \Sc(X) \to \Sc(X) \overset{r}{\to} \Sc(X_0)\to 0.$$

$$0 \ot \Sc(X_0) \overset{r}{\ot} \Sc(X)  \ot V \otimes \Sc(X)    \ot 0.$$

Let $v \in V$ be a generator. Then the complex is isomorphic to
%$$0 \to \Sc(X) \overset{m_h}{\to} \Sc(X) \overset{r}{\to} \Sc(X_0)\to 0,$$
$$0 \ot \Sc(X_0)\overset{r}{\ot}  \Sc(X) \overset{m_h}{\ot} \Sc(X)  \ot 0,$$

where $m_h(f)=hf$, and $h(x) = d\theta(\langle v, \phi(x)\rangle)$.
Clearly, $m_h$ is injective. The map $r$ is onto by Theorem \ref{thm:SchwarProp}\eqref{prop:ExtClose}
%Lemma \ref{lem:ExtClose}.
Note that $ih$ is a real-valued Nash function and 0 is its regular value. The exactness in the middle follows now from Lemma \ref{lem:DivSc}.

Induction step. Let us first prove (\ref{it:SimpAcyc}). Let $L <V$ be a one-dimensional subspace. Let $\phi_L$ denote the composition $X \to V^* \to L^*$. Note that 0 is a regular value of $\phi_L$ and let $X_{L,0}:=\phi_L^{-1}(0)$.
By the induction base, $\Sc(X)$ is $\fl$-acyclic and $r_L:\oH_0(\fl,\Sc(X)) \IsoTo \Sc(X_{L,0})$. Note that this is an isomorphism of representations of $V/L$.
Note also that $\phi(X_{L,0}) \subset (V/L)^* = L^{\bot} \subset V^*$ and let $\phi'$ denote $\phi|_{X_{L,0}}: X_{L,0} \to (V/L)^*$. Note that 0 is a regular value of $\phi'$.
Finally, note that the action of $V/L$ on $\Sc(X_{L,0})$ is $\tau_{X_L,\phi'}$.
Thus, by the induction hypothesis, $\oH_0(\fl,\Sc(X))$ is $V/L$-acyclic and Lemma \ref{lem:AcycComp} implies that $\Sc(X)$ is $V$-acyclic.

To prove (\ref{it:SimpH0}) note first that $r$ is onto by
%Lemma \ref{lem:ExtClose}.
Theorem \ref{thm:SchwarProp}\eqref{prop:ExtClose}.
Now we have to show that $\fv \Sc(X) = \Ker r$. The inclusion $\subset$ is obvious. Let $f \in \Ker r$. Consider $f|_{X_{L,0}}$. As before, the induction hypothesis implies $f|_{X_{L,0}} \in (\fv/\fl)\Sc(X_{L,0})$. Let $h \in (\fv/\fl)\otimes \Sc(X_{L,0})=\tau_{X_{L,0},\phi'}((\fv/\fl)\otimes \Sc(X_{L,0}))$ such that $f|_{X_{L,0}} = \tau_{X_{L,0},\phi'}(h)$. By Theorem \ref{thm:SchwarProp}\eqref{prop:ExtClose},%\ref{lem:ExtClose},
we may extend $h$ to $\widetilde{h} \in \fv\otimes \Sc(X)$. Now, $(\tau_{X,\phi}(\widetilde{h})-f)_{X_L} = 0$ and hence by the induction base $\tau_{X,\phi}(\widetilde{h})-f \in \fl \Sc(X)$. This implies that $f \in \fv \Sc(X)$. \proofend

\subsection{Proof of Lemma \ref{lem:0Reg}}\label{subsec:Pf0Reg}

For the proof we will need the following straightforward lemma from linear algebra.

\begin{lem}\label{lem:LinOnto}
Let $W_1,W_2$ and $W$ be linear spaces. Let $T:W_1 \onto W_2$ and $A:W_1 \onto W$ be epimorphisms.
Suppose that $\Ker T \supset \Ker A$. Then $A \circ S$ is onto, for any section $S$ of $T$.
\end{lem}

\begin{proof}[Proof of Lemma \ref{lem:0Reg}]
Let $x \in U$ such that $\mu(x)=0$. Without loss of generality, using the left action of $G$ we can suppose $s(x)=1$. Then $\phi|_L=0$ and we have to show $d_{1}a \circ d_{x}s$ is onto $L^*$.  The map $d_1a$ is onto by transitivity of the action of $G$ on $V^* \setminus 0$, and $\Ker d_1p \subset \Ker d_1a$ since $Q$ stabilizes $L$. Thus, by the previous lemma, $d_{1}a \circ d_{x}s$ is onto.
\end{proof}

\appendix

\section{Proof of the relative Shapiro lemma}\label{app:PfShapLem}
\setcounter{lemma}{0}

We prove here Theorem \ref{ShapLem} which is a version of the Shapiro lemma. A similar version was proven in \cite{AGRhamShap} and our proof follows the lines of the proof there.

\begin{proposition}[\cite{AGRhamShap}, Proposition 4.0.6]\label{QuotFib}
Let $G$ be a Nash group and $X$ be a Nash $G$-manifold. Suppose
that the action is strictly simple (see Definition \ref{def:StrSim}). Then the projection $\pi:X
\rightarrow G \setminus X$ is a Nash locally trivial fibration.
\end{proposition}
%
%\begin{corollary}[\cite{AGRhamShap}, Corollary 4.0.7]
%Let $G$ be a Nash group and $M$ be a Nash $G$ manifold with
%strictly simple action. Let $N$ be any $G$ manifold. Then the
%diagonal action on $M \times N$ is strictly simple.
%\end{corollary}

\begin{corollary}%[\cite{AGRhamShap}, Corollary 4.0.7]
Let $G$ be a Nash group and $X$ be a Nash $G$-manifold with
strictly simple action. Let $F \to X$ be a Nash  $G$-equivariant locally-trivial fibration. Then the action of $G$ on the total space $F$ is strictly simple.
\end{corollary}
\begin{proof}
By the previous proposition we may assume without loss of generality that the map $X \to G \setminus X$ is a trivial fibration. Hence we can identify $X \cong (G\setminus X) \times G$. Now it is easy to see that $G \setminus F \cong F|_{(G \setminus X) \times \{1\}}$.
\end{proof}

\begin{cor}\label{pulled}
Let $G$ be a Nash group and $X$ be a Nash $G$-manifold. Suppose
that the action is strictly simple. Let $\cE$ be a Nash $G$-equivariant bundle on $X$. Then there exists a Nash bundle $\cE'$ on $G \setminus X$ such that $\cE=\pi^*\cE'$, where $\pi:X \to G \setminus X$ is the standard projection.
\end{cor}
\begin{proof}
Consider the action of $G$ on the total space of $\cE$. Denote $\cE':=G \setminus \cE$ and consider it as a Nash bundle over $G \setminus X$. It is easy to see that $\cE=\pi^*\cE'$.
\end{proof}

\begin{defn}
Let $X$ be a Nash manifold and $F$ be a Nash locally trivial fibration over $X$. Denote by $H^i_c(F \to X)$ the natural Nash bundle on $X$ such that for any $x \in X$ we have  $H^i_c(F \to X)_x=H^i_c(F_x).$ For its precise definition see \cite[Notation 2.4.11]{AGRhamShap}.
Note that if $F$ is a trivial fibration then $H^i_c(F \to X)$ are trivial bundles.
\end{defn}

\begin{definition} %\label{}
Let $F \overset{\pi}{\rightarrow} X$ be a locally trivial
fibration. Let $ \cE \rightarrow X$ be a Nash bundle. We define
$T_{F \rightarrow X} \subset T_F$ by $T_{F \rightarrow X}=\Ker(d
\pi)$. We denote %
$$\Omega^{i,\cE}_{F \rightarrow X}:=((T_{F \rightarrow X})^*)^{\wedge i}\otimes \pi^*\cE  $$
Now we define \textbf{the relative de-Rham complexes}
$DR^{\cE}_{{\Sc}}(F \rightarrow X)$ and $DR^{\cE}_{{C^\infty}}(F \rightarrow X)$
by $$DR^{\cE}_{{\Sc}}(F \rightarrow X)^i:=\Sc(F, \Omega^{i,\cE}_{F \rightarrow X}) \text{ and }
DR^{\cE}_{{C^\infty}}(F \rightarrow X)^i:=C^\infty(F, \Omega^{i,\cE}_{F \rightarrow X}).$$
The differentials  in those complexes are defined as the differential in the classical de-Rham complex.
If $\cE$ is trivial we will omit it.
\end{definition}

\begin{theorem}[see \cite{AGRhamShap}, Theorem 3.2.3] \label{RelDeRham}
Let $\pi:F \rightarrow X$ be an affine Nash locally trivial fibration.
%Suppose that $F$ and $X$ are affine
Then
$$H^k(DR_{\Sc}^{\cE}(F\rightarrow X)) \cong {\Sc}(X,H^k_c(F\rightarrow
X)\otimes {\cE}).$$

\end{theorem}

This theorem gives us the following recipe for computing Lie algebra homology.

\begin{theorem}\label{recipe}
Let $G$ be an affine Nash group.  Let $K$ be a Nash $G$-manifold and $L$ be a Nash manifold. Let $X:=K \times L$. Let ${\cE}
\rightarrow X$ be a Nash $G$-equivariant bundle.
Let $\Chi:G \times L \to \R$ be a Nash map such that for any $l\in L$, the map $\Chi|_{G\times \{l\}}$ is a group homomorphism.
Let $\Chi'$ be the 1-dimensional $G$-equivariant bundle on $L$, with action of $G$ given by $g(l,v)=(l,\theta(\Chi(g,l))v)$.
Let $\cE' := \cE \otimes (\C \boxtimes \Chi')$. Note that $\cE'$ is isomorphic to $\cE$ as a Nash bundle, though
the $G$-equivariant structure on $\cE'$ is not necessarily Nash.
%$\cE'$ is not a Nash $G$-bundle.

Let $N$ be a strictly simple Nash $G$-manifold. Suppose that $N$ and $G$ are
homologically trivial %(i.e. all their homology except $\oH_0$ vanish and $\oH_0=\R$)
and affine. Denote $F=X\times N$ . Note that
the bundle $\cE' \boxtimes \Omega_N^i$ has $G$-equivariant
structure given by diagonal action. Hence the relative de-Rham
complex $DR_{C^\infty}^{\cE'}(F \rightarrow X)$ is a complex of
representations of $G$. Consider  the relative de-Rham
complex $DR_{\Sc}^{\cE}(F \rightarrow X)$ without the action of $G$ as a subcomplex of $DR_{C^\infty}^{\cE'}(F \rightarrow X)$. Note that it is $G$-invariant.
We denote by $DR_{\Sc}^{\cE'}(F \rightarrow X)$ the complex $DR_{\Sc}^{\cE}(F \rightarrow X)$  with the action of $G$ induced from $DR_{C^\infty}^{\cE'}(F \rightarrow X)$.

Then $$\oH_i(\mathfrak{g},
\Sc(X,\cE')) = H^{n-i}((DR_{\Sc}^{\cE'}(F \rightarrow X))_{\mathfrak{g}}),$$
where $n$ is the dimension of $X$.
\end{theorem}

To prove this theorem we will need the following statements.
\begin{lemma} \label{uf}
Let $G$, $L$, $N$, $\Chi$, $\Chi'$ be as in Theorem \ref{recipe}. Let $X$ be a Nash manifold. Let $\cE$ be a Nash bundle over $X \times L\times N$. Let $\cE'= \cE \otimes (\C\boxtimes \Chi' \boxtimes \C)$. Then $\Sc(X \times L\times N,\cE) \simeq \Sc(X \times L\times N,\cE')$ as representations of $G$.
\end{lemma}
\begin{proof}
By Theorem \ref{thm:SchwarProp}\eqref{pCosheaf} and Proposition \ref{QuotFib} we can assume that $N=G \times N'$, for some Nash manifold $N'$. Hence we can assume $N=G$. Note that
$\Sc(X \times L\times G,\cE)$ and $\Sc(X \times L\times G,\cE')$ are identical as linear spaces. Now, the required isomorphism between them is given by $\alp \mapsto (1 \boxtimes (\theta \circ \Chi)) \alp$.
\end{proof}

\begin{lemma} \label{DRgroup}
Let $G$ be an affine Nash group. Let $F$ be a strictly simple Nash
$G$-manifold. Denote $X:=G \setminus F$. Let $\cE \rightarrow X$ be a
Nash bundle. Then the relative de-Rham complex
$DR_{\Sc}^\cE(F\rightarrow X)^i$ is isomorphic to the complex
$\widetilde{C}(\mathfrak{g},\Sc(F,\pi^*\cE))^{\dim \g-i}$, where $\pi:F\rightarrow X$ is the
standard projection, %$d$ is the dimension of fibers of $F$
and $\widetilde{C}(\mathfrak{g},W)$ denotes the Koszul complex of a representation $W$.
%standard complex that computes homology of a representation $V$ of $\g$.
\end{lemma}
The proof of this lemma is the same as the proof of \cite[Lemma 4.0.10]{AGRhamShap}.

\begin{cor}\label{Lokatavnu}
Let an affine Nash group $G$ act strictly simply on a Nash manifold $X$. Let $\cE$ be a Nash $G$-equivariant bundle on $X$. Suppose that $G$ is homologically trivial.
Then $\Sc(X,\cE)$ is an acyclic representation of $\g$.
\end{cor}
\begin{proof}
Follows from Lemma \ref{DRgroup}, Theorem \ref{RelDeRham} and Corollary \ref{pulled}.
\end{proof}

\begin{proof}[Proof of Theorem \ref{recipe}] From Theorem \ref{RelDeRham} we know that the complex $DR_{\Sc}^\cE(F
\rightarrow X)$, as a complex of vector spaces, is a resolution of the vector space $\Sc(X,\cE)$.
Hence the complex $DR_{\Sc}^{\cE'}(F\rightarrow X)$ is a resolution of $\Sc(X,\cE')$ as a representation of $\g$.

So it is enough to prove that the representations $\Sc(F,\cE'
\boxtimes \Omega_N^i)$ are $\mathfrak{g}$- acyclic. By Lemma \ref{uf}, $$\Sc(F,\cE
\boxtimes \Omega_N^i) \cong \Sc(F,\cE'
\boxtimes \Omega_N^i)$$ as $\g$-representations (though those isomorphisms do not commute with differentials).
Hence it is enough to show that $\Sc(F,\cE
\boxtimes \Omega_N^i)$ is acyclic. This follows from Corollary \ref{Lokatavnu}.
\end{proof}
\begin{lem}[\cite{AGST}, Corollary B.1.8]
Let $G$ be a connected affine Nash group, $X$ be a Nash manifold and $\cE$ be a Nash bundle over $X$. Let $\g$ be the Lie algebra of $G$. Let $p:G\times X \to X$ be the projection. Let $G$ act on $\Sc(G \times X,p^!(\cE))$ by acting on the $G$ coordinate. Then
%$\g \Sc(G \times X,p^?(\cE)) = \Sc(G \times X,p^?(\cE))_{0,X}.$
the map $p_*:\Sc(G \times X,p^!(\cE)) \to \Sc(X,\cE)$ induces an isomorphism $\Sc(G \times X,p^!(\cE))_{\g} \cong \Sc(X,\cE)$.
\end{lem}
\begin{cor} \label{DRgroupCor}
Let $G$ be a connected affine Nash group. Let $F$ be a strictly simple Nash
$G$-manifold. Denote $X:=G \setminus F$.
Let $\cE \rightarrow F$ be a $G$-equivariant
Nash bundle.
%Fix a Haar measure on $G$ and a Nash measure on $X$. These give rise to a Nash measure on $F$.
Let $\pi$ denote the projection map $\pi:F \to X$. Let $B$ be a Nash bundle on $X$ such that $\cE = \pi^*B$ ($B$ exists by Corollary \ref{pulled}).

Then the map $\pi_*:\Sc(F,\cE \otimes D^F_X ) \to \Sc(X,B)$ induces an isomorphism $\Sc(F,\cE\otimes D^F_X )_{\g} \cong \Sc(X,B)$.
\end{cor}

\begin{cor}
Let $X$ be a Nash manifold and $Y$ be a closed Nash submanifold. Let an affine Nash group $G$ act on $X$ strictly simply and $H$ be a subgroup of $G$ that acts on $Y$ strictly simply. Suppose that the natural map $H \setminus Y \to G \setminus X$ is a Nash diffeomorphism.
Suppose that $G$ and $H$ are homologically trivial.
Let $\cE$ be a Nash $G$-equivariant bundle on $X$.
Then
$$\Sc(X,\cE\otimes D_X)_{\g} \cong \Sc(Y,\cE|_Y \otimes D_Y)_{\h}.$$
Moreover, let $a:G \times Y \to X$ denote the action map and $p_2:G \times Y \to Y$ denote the projection on the second coordinate.
Then\\
%\begin{multline}
(i) The map $a_*$ gives an isomorphism
$$a_*:\Sc(G \times Y, (\C \boxtimes \cE|_Y) \otimes D_{G \times Y} \otimes a^*(D_X^{-1}))_{\g \times \h} \IsoTo \Sc(X,\cE)_{\g}$$
(ii) The map $(p_2)_*$ gives an isomorphism
$$(p_2)_*:\Sc(G \times Y, (\C \boxtimes \cE|_Y) \otimes D_{G \times Y} \otimes a^*(D_X^{-1}))_{\g \times \h} \IsoTo \Sc(Y,\cE|_Y \otimes D_Y \otimes ((D_X)|_Y)^{-1})_{\h}$$
%\end{multline}
\end{cor}
%\begin{proof}
%The lemma follows from Lemma \ref{DRgroupCor}.
%\end{proof}

\begin{cor}\label{H2G}
Let $G$ be an affine Nash group and $X$ be a transitive Nash $G$-manifold.
Let $N$ be a Nash manifold.
Let $x \in X$ and denote $H:=G_x$. Suppose that $G$ and $H$ are homologically trivial.
Let $X = X \times N \times G$ and $Y=\{x\} \times N \times G$. Let $\cE$ be a bundle over $X$.
Let $\Chi:G \times N \to \R$ be a Nash map such that for any $n\in N$, the map $\Chi|_{G\times \{n\}}$ is a group homomorphism.
Let $\Chi'$ be the 1-dimensional $G$-equivariant bundle on $N$, with the action of $G$ given by $g(n,v)=(n,\theta(\Chi(g,l))v)$.
Let $\cE' := \cE \otimes (\C \boxtimes \Chi' \boxtimes \C)$.
%Then
%(i) The map $a_*$ gives an isomorphism $a_*:\Sc(G \times Y, \C \boxtimes \cE')_{\g \times \h} \cong \Sc(X,\cE')_{\g}$\\
%(ii) projection on  the second coordinate  gives an isomorphism
%$\Sc(G \times Y, \C \boxtimes \cE')_{\g \times \h} \cong \Sc(Y,\cE'|_Y)_{\h}$.
Then\\
%\begin{multline}
(i) The map $a_*$ gives an isomorphism
$$a_*:\Sc(G \times Y, (\C \boxtimes \cE'|_Y) \otimes D_{G \times Y} \otimes a^*(D_X^{-1}))_{\g \times \h} \cong \Sc(X,\cE')_{\g}$$
(ii) The map $(p_2)_*$ gives an isomorphism
$$(p_2)_*:\Sc(G \times Y, (\C \boxtimes \cE'|_Y) \otimes D_{G \times Y} \otimes a^*(D_X^{-1}))_{\g \times \h} \cong \Sc(Y,\cE'|_Y \otimes D_Y \otimes ((D_X)|_Y)^{-1})_{\h}$$
\end{cor}

Now we are ready to prove the relative Shapiro Lemma.

\begin{proof} [Proof of Theorem \ref{ShapLem}.]
From Theorem \ref{recipe}  we see that
$$\oH_i(\mathfrak{g},\Sc(X \times N,\cE'))\cong H^{\dim \g-i}((DR^{\cE'}_{\Sc}(X\times N \times
G\rightarrow X \times N))_{\mathfrak{g}}).$$

Now, by Corollary \ref{H2G},
$$\Sc(X\times N \times G, \Omega^{i,\cE'}_{X \times  N \times G \rightarrow X \times N })_\mathfrak{g} \cong
\Sc(\{x\} \times N \times G, \Omega^{i,\cE'|_{\{x\}\times N} \otimes D_{\{x\}\times N} \otimes D_{X\times N}^{-1}|_{\{x\}\times N}}_{\{x\} \times N \times G\rightarrow \{x\} \times N}
)_{\mathfrak{h}}.$$
and this isomorphism commutes with de-Rham differential. Therefore%
$$(DR^{\cE'}_{\Sc}(X\times N \times G\rightarrow X \times N))_\mathfrak{g} \cong
(DR^{\cE'|_{\{x\}\times N} \otimes D_{\{x\}\times N} \otimes D_{X\times N}^{-1}|_{\{x\}\times N}}_{\Sc}(\{x\} \times N \times  G\rightarrow \{x\}\times N))_{\mathfrak{h}}$$
and hence
$$H^{\dim \g-i}((DR^{\cE'}_{\Sc}(X\times N \times G\rightarrow X \times N))_\mathfrak{g}) \cong
H^{\dim \g-i}((DR^{\cE'|_{\{x\}\times N}\otimes (D_{X}^{-1}|_{\{x\}}\boxtimes \C)}_{\Sc}(\{x\} \times N \times  G\rightarrow \{x\}\times N))_{\mathfrak{h}})$$
and again by Theorem \ref{recipe},
\begin{multline*}
H^{\dim \g-i}((DR^{\cE'|_{\{x\}\times N} \otimes (D_{X}^{-1}|_{\{x\}}\boxtimes \C)}_{\Sc}(\{x\} \times N \times  G\rightarrow \{x\}\times N))_{\mathfrak{h}}) \cong \\
\cong \oH_i(\mathfrak{h},\Sc(\{x\}\times N,\cE'|_{\{x\}\times N}
\otimes (D_{X}^{-1}|_{\{x\}}\boxtimes \C))) \cong \\ \cong  \oH_i(\mathfrak{h},\Sc(\{x\}\times N,\cE'|_{\{x\}\times N} \otimes \Delta_{H} \cdot (\Delta_G^{-1})|_H)).
\end{multline*}
\end{proof}

\section{Examples for the notation of \S \ref{subsec:ProdGeo}}\label{app:ProdEx}
In order to help the reader to read Notations \ref{not:Orbits1} and \ref{not:Orbits2},  Lemma \ref{lem:SysRep} and Corollary \ref{cor:SysRep}, we describe explicitly the objects discussed in them on one example.
Let $n=6$ and $\lambda = (1,3,2)$. Then $m_{\lambda}^1 = 1$, $m_{\lambda}^2 = 4$, $m_{\lambda}^3 = 6$. Also,
$c$ is the permutation $(234561), \, w^1_{\lambda}=c^0=Id_6, \, w^2_{\lambda}=c^2=(345612),\, w^3_{\lambda}=c^5=(612345)$.
We have $$\overline{P_n w^1_{\lambda} P_{\lam}}=P_n w^1_{\lambda} P_{\lam}, \overline{P_n w^2_{\lambda} P_{\lam}}=P_n w^2_{\lambda} P_{\lam} \cup P_n w^1_{\lambda} P_{\lam}, \, \overline{P_n w^3_{\lambda} P_{\lam}}= P_n w^3_{\lambda} P_{\lam} \cup P_n w^2_{\lambda} P_{\lam} \cup P_n w^1_{\lambda} P_{\lam} =G_6, \, P_{\lambda}^1 = P_{\lam}$$
Typical matrices in $P_{\lam},\, \overline{P_n w^1_{\lambda} P_{\lam}},\, \overline{P_n w^2_{\lambda} P_{\lam}},\, P_{\lambda}^2 $ are of the forms
$$\left(
  \begin{array}{cccccc}
    * & * & * & * & * & * \\
    0 & * & * & * & * & * \\
    0 & * & * & * & * & * \\
    0 & * & * & * & * & * \\
    0 & 0 & 0 & 0 & * & * \\
    0 & 0 & 0 & 0 & * & * \\
  \end{array}
\right), \,
%
% c = \left(
%   \begin{array}{cccccc}
%     0 & 0 & 0 & 0 & 0 & 1 \\
%     1 & 0 & 0 & 0 & 0 & 0 \\
%     0 & 1 & 0 & 0 & 0 & 0 \\
%     0 & 0 & 1 & 0 & 0 & 0 \\
%     0 & 0 & 0 & 1 & 0 & 0 \\
%     0 & 0 & 0 & 0 & 1 & 0 \\
%   \end{array}
% \right)$\newline
% $ $\newline\\
% %
% $w_1 =Id_6, \quad
% %
% w_2=\left(
%   \begin{array}{cccccc}
%     0 & 0 & 0 & 0 & 1 & 0 \\
%     0 & 0 & 0 & 0 & 0 & 1 \\
%     1 & 0 & 0 & 0 & 0 & 0 \\
%     0 & 1 & 0 & 0 & 0 & 0 \\
%     0 & 0 & 1 & 0 & 0 & 0 \\
%     0 & 0 & 0 & 1 & 0 & 0 \\
%   \end{array}
% \right), \quad
% w_3=\left(
%   \begin{array}{cccccc}
%     0 & 1 & 0 & 0 & 0 & 1 \\
%     0 & 0 & 1 & 0 & 0 & 0 \\
%     0 & 0 & 0 & 1 & 0 & 0 \\
%     0 & 0 & 0 & 0 & 1 & 0 \\
%     0 & 0 & 0 & 0 & 0 & 1 \\
%     1 & 0 & 0 & 0 & 0 & 0 \\
%   \end{array}
% \right)$\\\\
%$ $\\
\left(
  \begin{array}{cccccc}
    * & * & * & * & * & * \\
    * & * & * & * & * & * \\
    * & * & * & * & * & * \\
    * & * & * & * & * & * \\
    * & * & * & * & * & * \\
    0 & 0 & 0 & 0 & * & * \\
  \end{array}
\right), \,
\left(
  \begin{array}{cccccc}
    * & * & * & * & * & * \\
    * & * & * & * & * & * \\
    * & * & * & * & * & * \\
    * & * & * & * & * & * \\
    * & * & * & * & * & * \\
    0 & * & * & * & * & * \\
  \end{array}
\right),\, \left(
  \begin{array}{cccccc}
    * & * & 0 & 0 & 0 & 0 \\
    * & * & 0 & 0 & 0 & 0 \\
    * & * & * & * & * & * \\
    * & * & 0 & * & * & * \\
    * & * & 0 & * & * & * \\
    * & * & 0 & * & * & * \\
  \end{array}
\right)$$
%\begin{multline}
%

%
Typical matrices in $P^{3}_{\lam},\, Q_{\lam}^1,Q^2_{\lam},\, Q_{\lambda}^3 $ are of the forms
$$\left(
  \begin{array}{cccccc}
    * & * & * & * & * & 0 \\
    * & * & * & * & * & 0 \\
    * & * & * & * & * & 0 \\
    0 & 0 & 0 & * & * & 0 \\
    0 & 0 & 0 & * & * & 0 \\
    * & * & * & * & * & * \\
  \end{array}
\right),\,
\left(
  \begin{array}{cccccc}
    * & * & * & * & * & * \\
    0 & * & * & * & * & * \\
    0 & * & * & * & * & * \\
    0 & * & * & * & * & * \\
    0 & 0 & 0 & 0 & * & * \\
    0 & 0 & 0 & 0 & 0 & 1 \\
  \end{array}
\right) ,\,
\left(
  \begin{array}{cccccc}
    * & * & 0 & 0 & 0 & 0 \\
    * & * & 0 & 0 & 0 & 0 \\
    * & * & * & * & * & * \\
    * & * & 0 & * & * & * \\
    * & * & 0 & * & * & * \\
    0 & 0 & 0 & 0 & 0 & 1 \\
  \end{array}
\right),\, \left(
  \begin{array}{cccccc}
    * & * & * & * & * & 0 \\
    * & * & * & * & * & 0 \\
    * & * & * & * & * & 0 \\
    0 & 0 & 0 & * & * & 0 \\
    0 & 0 & 0 & * & * & 0 \\
    0 & 0 & 0 & 0 & 0 & 1 \\
  \end{array}
\right) $$
%\end{multline}

In order to help the reader to read Lemma \ref{lem:ProdGeo}, we describe explicitly the objects discussed there for $i=2$.
Let $\lambda'=\lambda_{3-2+1}^-=\lambda_2^-= (1, 2,2)$. Then
$w_{\lambda'}^{21}$ is the permutation $(45123),w_{\lambda'}^{22}=Id,
w_{\lambda'}^{23}=(34512)$


\begin{thebibliography}{999999}                                                                                           %
\bibitem[\href{http://imrn.oxfordjournals.org/cgi/reprint/2008/rnm155/rnm155?ijkey=bddq0itkXKrVjlG&keytype=ref}{AG08}]{AGSc} A. Aizenbud,  D.
Gourevitch: {\it Schwartz functions on Nash Manifolds,}
International Mathematics Research Notices, Vol. 2008, n.5,
Article ID rnm155, 37 pages. DOI: 10.1093/imrn/rnm155. See also
arXiv:0704.2891 [math.AG].
%\bibitem[\href{http://arxiv.org/PS_cache/arxiv/pdf/0802/0802.3305v2.pdf}{AG10}]{AGRhamShap} Aizenbud, A.; Gourevitch, D.: {\it De-Rham theorem and Shapiro lemma for Schwartz functions on Nash manifolds.}
%To appear in the Israel Journal of Mathematics. See also
%arXiv:0802.3305v2 [math.AG].
%%\bibitem[Abe]{Abe}N. Abe: \textit{Generalized Jacquet modules of parabolic
%%induction}, arXiv:0710.3206v2[math.RT].
%
%\bibitem[AG09]{AG}A. Aizenbud, D. Gourevitch, \textit{Multiplicity one
%theorem for $(GL_{n+1}(\mathbb{R}),GL_{n}(\mathbb{R}))$}, Selecta Mathematica,
%\textbf{15}, No. 2, pp. 271-294 (2009). See also arXiv:0808.2729[math.RT].

\bibitem[\href{http://arxiv.org/PS_cache/arxiv/pdf/0802/0802.3305v2.pdf}{AG10}]{AGRhamShap}A. Aizenbud,  D.
Gourevitch: {\it De-Rham theorem and Shapiro lemma for Schwartz functions on Nash manifolds.} Israel Journal of Mathematics  \textbf{177}, Number 1, 155-188 (2010). See also arXiv:0802.3305v2 [math.AG].

\bibitem[\href{http://arxiv.org/abs/0812.5063}{AG09}]{AGHC} A. Aizenbud, D.
Gourevitch:{\it Generalized Harish-Chandra descent, Gelfand pairs and an
archimedean analog of Jacquet-Rallis' Theorem,} Duke Mathematical Journal   \textbf{149}, Number 3, 509-567 (2009). See also arXiv: 0812.5063[math.RT].

\bibitem[AG]{AGST} A. Aizenbud,  D. Gourevitch: {\it
Smooth Transfer of Kloostermann Integrals}, American Journal of Mathematics \textbf{135}, 143-182 (2013). See also arXiv:1001.2490[math.RT].

\bibitem[AOS12]{AOS}A. Aizenbud, O. Offen, E. Sayag: \textit{Disjoint pairs for $GL_n(\R)$ and $GL_n(\C)$}, C. R. Math. Acad. Sci. Paris \textbf{350}, no. 1-2, pp 9–11 (2012).

\bibitem[AGS]{AGS1} A. Aizenbud,  D. Gourevitch, S. Sahi: {\it
Derivatives for smooth representations of $GL(n,{\mathbb{R}})$ and $GL
(n,{\mathbb{C}})$}, arXiv:1109.4374v2.

%\bibitem[AS]{ASGel}A. Aizenbud, E. Sayag: \textit{Invariant distributions on
%non-distinguished nilpotent orbits with application to the Gelfand property of
%(GL(2n,R),Sp(2n,R))}, arXiv:0810.1853 [math.RT].
%\bibitem[AOS]{AOSUnDisj}A. Aizenbud, O. Offen, E. Sayag: \textit{Uniqueness
%and disjointness of Klyachko models in the Archimedean case}, in preparation.

%\bibitem[Bar03]{Bar}E.M. Baruch, \textit{A proof of Kirillov's conjecture},
%Annals of Mathematics, 158, 207-252 (2003).
%
%\bibitem[BB82]{BB1}W. Borho, J.-L. Brylinski: \textit{Differential operators
%on homogeneous spaces, I,} Invent.
%Math. \textbf{69} , pp. 437-476 (1982).
%
%\bibitem[BB89]{BB2}W. Borho, J.-L. Brylinski \textit{Differential operators on
%homogeneous spaces. II. Relative enveloping algebras.} Bull. Soc. Math. France
%\textbf{117} (1989), no. 2, 167--210.

\bibitem[BCR98]{BCR} J. Bochnak, M. Coste, M-F. Roy:
{\it Real Algebraic Geometry} Berlin: Springer, 1998.

%\bibitem[Ber84]{Ber}J. Bernstein: \textit{$P$-invariant Distributions on
%$\mathrm{GL}(N)$ and the classification of unitary representations of
%$\mathrm{GL}(N)$ (non-archimedean case),} Lie group representations, II
%(College Park, Md., 1982/1983), 50--102, Lecture Notes in Math.,
%\textbf{1041}, Springer, Berlin (1984).

%\bibitem[BSS90]{BSS}D. Barbasch, S. Sahi,  and B. Speh: \textit{Degenerate
%series representations for GL(2n, R) and Fourier analysis}. Symposia
%Mathematica, Vol.\ XXXI (Rome, 1988) 45-69 (1990).


%\bibitem[BV78]{BV}D. Barbasch, and J. Vogan: \textit{The local structure of
%characters. Journal of Functional Analysis} \textbf{37}, 27-55 (1980).
%
%\bibitem[BK]{BerKr}J. Bernstein, and B. Kroetz: \textit{Smooth Frechet Globalizations of Harish-Chandra Modules,} arXiv:0812.1684.

\bibitem[BZ77]{BZ-Induced}I.N. Bernstein, A.V. Zelevinsky: \textit{Induced
representations of reductive p-adic groups. I.} Ann. Sci. Ec. Norm. Super,
4$^{\text{e}}$serie \textbf{10}, 441-472 (1977).

\bibitem[Cas89]{CasGlob} W. Casselman:  \textit{Canonical extensions of Harish-Chandra modules to representations of G,}
Can. J. Math.,  \textbf{XLI}, No. 3,  385-438 (1989).

\bibitem[CHM00]{CHM}W. Casselman, H. Hecht, D. Mili\v ci\'c:
\textit{Bruhat filtrations and Whittaker vectors for real groups}. The
mathematical legacy of Harish-Chandra (Baltimore, MD, 1998), 151-190, Proc.
Sympos. Pure Math., \textbf{68}, Amer. Math. Soc., Providence, RI, (2000)

\bibitem[CM82]{CM}W. Casselman, D. Mili\v ci\'c:
\textit{Asymptotic behavior of matrix coefficients of admissible representations}.
Duke Math. J.  \textbf{49}, Number 4, 869-930 (1982).

%\DimaA{
%\bibitem[CPS]{CPS} J. Cogdell and I.I. Piatetski-Shapiro: {\it Derivatives and L-functions for GL(n) }, available at \url{http://www.math.osu.edu/~cogdell/moish-www.pdf}.  To appear in {\it The Heritage of B. Moishezon}, IMCP.}
%
%\bibitem[CoMG93]{CoMG}D. Collingwood, W. McGovern: \textit{Nilpotent orbits in
%semisimple Lie algebras.} Van Nostrand Reinhold Mathematics Series. Van
%Nostrand Reinhold Co., New York (1993). xiv+186 pp.




\bibitem[GM88]{GM} S. Gelfand, Y. Manin: \textit{Metody gomologicheskoi algebry. Tom 1. [Methods in homological algebra. Vol. 1] Vvedenie v teoriyu kogomologii i proizvodnye kategorii. [Introduction to the theory of cohomology, and derived categories]}, with an English summary. ``Nauka'', Moscow, 1988. 416 pp.

%\bibitem[GaLa]{GaLa}E. Galina, Y. Laurent: \textit{Kirillov's conjecture and
%D-modules,} arXiv:0910.2900.


%\bibitem[GS]{GS}D. Gourevitch, S. Sahi:
%\textit{Associated varieties, derivatives, Whittaker functionals, and
%rank for unitary representations of $GL(n)$}, Selecta Mathematica (New Series), online first (2012), DOI:10.1007/s00029-012-0100-8,
%arXiv:1106.0454.
%
%
%\bibitem[GS2]{GS-Gen}D. Gourevitch, S. Sahi: \textit{Degenerate Whittaker models
%for real reductive groups}, preprint.


%\bibitem[GOSS]{GOSS}D. Gourevitch, O. Offen, S. Sahi and E. Sayag:
%\textit{Existence of Klyachko models for $GL(n,\mathbb{R})$ and
%$GL(n,\mathbb{C})$}, preprint.
%available at \url{http://www.wisdom.weizmann.ac.il/~dimagur/Publication_list.htm}.

%\bibitem[He08]{He}H. He: \textit{Associated varieties and Howe's N-spectrum},
%Pacific Journal of Mathematics, \textbf{237} No. 1 (2008).

%\bibitem[How82]{How}Howe, R. On a notion of rank for unitary representations
%of the classical groups. Harmonic analysis and group representations 223-331 (1982).
%
%\bibitem[Jan77]{Jan} C.J. Jantzen: \textit{Kontravariente Former auf induzierten Darnstellungen halbeinfacher Lie-Algebren}, Math. Ann., \textbf{226}, 53-65 (1977).
%
%\bibitem[Jos80]{Jos81}A. Joseph: \textit{Application de la theorie des anneaux
%aux algebres enveloppantes}. Lecture Notes, Paris (1980).
%
%\bibitem[Jos85]{Jos85}A. Joseph: \textit{On the associated variety of a
%primitive ideal}, Journal of Algebra \textbf{93} , no. 2, 509--523 (1985).

\bibitem[Kos78]{Kos}B. Kostant: \textit{On Whittaker vectors and representation theory.}, Invent. Math. \textbf{48}, 101-184 (1978).

%\bibitem[KR71]{KR}B. Kostant, S. Rallis: \textit{ Orbits and representations
%associated with symmetric spaces.} Amer. J. Math. \textbf{93} (1971), 753--809.

%\bibitem[KrLe99]{KrLe}G. R. Krause, T. H. Lenagan \textit{Growth of algebras
%and Gelfand-Kirillov dimension}. Revised edition. Graduate Studies in
%Mathematics, 22. American Mathematical Society, Providence, RI, 2000. x+212 pp.

%\bibitem[Mat87]{Mat}H. Matumoto: \textit{Whittaker vectors and associated
%varieties}, Invent. math. \textbf{89}, 219-224 (1987)

%\bibitem[Mat90]{MatDuke}H. Matumoto: \textit{ Whittaker modules associated
%with highest weight modules}, Duke Math. J. \textbf{60} (1990), no. 1, 59--113.

%\bibitem[MW87]{MW}C. Moeglin, J.L. Waldspurger: \textit{Modeles de Whittaker
%degeneres pour des groupes p-adiques,} Math. Z. \textbf{196} (1987), no. 3, pp 427-452.

%\bibitem[Nel59]{Nel}E. Nelson, \textit{Analytic vectors}, Ann. of Math.
%\textbf{70} (1959), 572-615.

%\bibitem[Pou72]{Pou}N. S. Poulsen, \textit{On $C^{\infty}$-vectors and
%Intertwining Bilinear Forms for Representations of Lie Groups}, Journal Of
%Functional Analysis \textbf{9}, 87-120 (1972)

%\bibitem[Ros95]{Ros}Rossmann, W. Picard-Lefschetz theory and characters of a
%semisimple Lie group. Inventiones Mathematicae 121, 579-611 (1995).
%\bibitem[Rud73]{Rud} W. Rudin :{\it Functional analysis} New York : McGraw-Hill, 1973.

\bibitem[Sah89]{Sahi-Kirillov}S. Sahi: \textit{On Kirillov's conjecture for
Archimedean fields,} Compositio Mathematica \textbf{72}, 67-86 (1989).
%
%\bibitem[Sah90]{Sahi-PAMS}S. Sahi: \textit{A simple construction of Stein's
%complementary series representations,} Proc. Amer. Math. Soc. \textbf{108}
%(1990), no. 1, 257--266.
%
%\bibitem[Sah95]{SahiJordan} S. Sahi: {\it Jordan algebras and degenerate principal series},
%J. reine angew. Math. \textbf{462}(1995),  1-18.

\bibitem[Shi87]{Shi} M. Shiota: {\it Nash Manifolds},  Lecture Notes in Mathematics \textbf{1269} (1987).

%\bibitem[SaSt90]{SaSt}S. Sahi, E. Stein: \textit{Analysis in matrix space and
%Speh's representations,} Invent. Math. \textbf{101}, no. 2, pp 379--393 (1990).

%\bibitem[Sca90]{Sca}R. Scaramuzzi, \textit{A Notion of Rank for Unitary
%Representations of General Linear Groups}, Transactions of the American
%Mathematical Society, \textbf{319}, No. 1 (1990), pp.349-379.

%\bibitem[Sek87]{Sek}Sekiguchi, J. \textit{Remarks on real nilpotent orbits of
%a symmetric pair}. Journal of the Mathematical Society of Japan 39, 127-138 (1987).

%\bibitem[SV00]{SV}W. Schmid and K. Vilonen, \textit{Characteristic cycles and
%wave front cycles of representations of reductive Lie groups}. Annals of
%Mathematics, \textbf{151} (2000), 1071-1118.
%
%\bibitem[SZ]{SZ}B. Sun and C.-B. Zhu, \textit{Multiplicity one theorems: the
%archimedean case}, arXiv:0903.1413[math.RT].

%\bibitem[Tad86]{Tad}M. Tadic, \textit{Classification of unitary
%representations in irreducible representations of general linear group
%(non-Archimedean case)}. Ann. Sci. Ecole Norm. Sup. (4) \textbf{19} (1986),
%no. 3, 335--382.

%\bibitem[Tra86]{Trapa}
%P.E. Trapa: \textit{Annihilators and associated varieties of $A_q(\lambda)$ modules for $U(p,q)$},
%Compositio Math. \textbf{129} , no. 1, 145 (2001).

%\bibitem[Vog78]{Vog}D. A. Vogan \textit{Gelfand-Kiriliov Dimension for
%Harish-Chandra Modules}, Inventiones math \textbf{48}, 75-98 (1978)

%\bibitem[Vog86]{Vog-class}D. A. Vogan: \textit{The unitary dual of GL(n) over
%an Archimedean field, }Inventiones math \textbf{83}, 449-505 (1986)
%
%\bibitem[Vog81]{Vog-book}D. A. Vogan: \textit{Representations of real reductive
%groups, }Birkhauser, Boston-Basel-Stuttgart, (1981)

\bibitem[Wall88]{Wal1}N. Wallach: \textit{Real Reductive groups I}, Pure and
Applied Math. \textbf{132}, Academic Press, Boston, MA (1988).

\bibitem[Wall92]{Wal2}N. Wallach: \textit{Real Reductive groups II}, Pure and
Applied Math. \textbf{132}, Academic Press, Boston, MA (1992).

%\bibitem[Zel80]{Zl}A.V. Zelevinsky : \textit{ Induced representations of
%reductive p-adic groups. II. On irreducible representations of Gl(n)}. Ann.
%Sci. Ec. Norm. Super, $4^{e}$serie \textbf{13}, 165-210 (1980)
\end{thebibliography}
\end{document}